%% file: main.tex
\documentclass{article}
\usepackage{style}
\usepackage{empheq}
\usepackage{natbib}
\usepackage{array}
\usepackage{subfig}
\usepackage{tikz}
\title{On the non-convexity issue in the radial Calder\'on problem}
\author{Giovanni S. Alberti\thanks{Machine Learning Genoa Center (MaLGa), Department of Mathematics, Department of Excellence 2023–2027, University of Genoa (\url{giovanni.alberti@unige.it}).} \and Romain Petit\thanks{CNRS and DMA, ENS, PSL Universit\'e (\url{romain.petit@ens.fr}).} \and Clarice Poon\thanks{Mathematics Institute, University of Warwick (\url{clarice.poon@warwick.ac.uk}
).} \and Irène Waldspurger\thanks{CNRS, Université Paris Dauphine, INRIA Mokaplan (\url{waldspurger@ceremade.dauphine.fr}).}}

\date{November 24, 2025}

\begin{document}
\maketitle
    
\begin{abstract}
A classical approach to the Calder\'on problem is to estimate the unknown conductivity by solving a nonlinear least-squares problem. It leads to a nonconvex optimization problem which is generally believed to be riddled with bad local minimums. We revisit this issue in the case of piecewise constant radial conductivities and prove that, contrary to previous claims, there are no spurious critical points in the case of two scalar unknowns with no measurement noise. We also provide a partial proof of this result in the general setting which holds under a numerically verifiable assumption. Finally, we investigate whether a recently proposed approach based on convexification yields better reconstructions. For the first time, we propose a way to implement it in practice and show that it is consistently outperformed by some least squares solvers, which are also faster and require less measurements.
\end{abstract}

\section{Introduction}
\subsection{Electrical impedance tomography and the Calder\'on problem}
\label{subsec_eit_calderon}
Electrical impedance tomography is a non-invasive imaging technique that consists in reconstructing the electrical conductivity of a medium from current-voltage boundary measurements. It was formalized in \cite{calderon_inverse_1980} as an inverse problem for a partial differential equation (PDE), and is hence often called the \emph{Calder\'on problem}. More precisely, the unknown of this inverse problem is modeled as a positive bounded function $\gamma:\Omega\to\RR$ in a domain $\Omega$. The application of an electrical current $g$ on the boundary of the domain induces an electric potential $u$ in $\Omega$ which is the unique solution to the following conductivity equation:
\begin{equation*}
    \begin{cases}
        \mathrm{div}(\gamma\nabla u)=0 &\mathrm{in}~\Omega,\\
        \gamma \partial_{\nu}u=g &\mathrm{on}~\partial\Omega.
    \end{cases}
\end{equation*}
The Calder\'on problem consists in recovering $\gamma$ from the current-to-voltage map (also called Neumann-to-Dirichlet map) $\Lambda(\gamma):g\mapsto \restriction{u}{\partial\Omega}$.

Although the map $\Lambda$ can be shown to be injective (given sufficiently many measurements, see the references in the section below), it is highly ill-posed. As only boundary measurements are available, large variations of the unknown $\gamma$ away from the boundary might result in very small variations of the measurements $\Lambda(\gamma)$.  Coupled with the issue that the forward map $\Lambda$ is nonlinear, the numerical resolution of the inverse problem is difficult, as reconstruction algorithms are potentially prone to the presence of local minimums and they might have to be carefully initialized. It is often mentioned that, due to the nonlinearity, the classical least squares approach leads to a nonconvex optimization problem which is riddled with  bad local minimums. Some recent references where this claim was reiterated include \cite{lazzaro2024oracle,brojatsch2024required,klibanov2025convexification,harrach2025monotonicity}. However, there is a lack of precise characterization in the literature on the nature of this non-convexity. 

In this work, we revisit the non-convexity issue in the case of piecewise constant radial conductivities. We find that some least squares solvers always converge to a global minimizer. We provide a proof of the absence of spurious critical points in the case of two scalar unknowns. We also prove this result in the general case under a numerically verifiable assumption. In the absence of noise, our numerical simulations suggest that local minimums are not present and that the nonconvexity issue for the Calder\'on problem is far more nuanced than the literature suggests. Instead, the real challenge seems to be the ill-posedness of the inverse problem. Motivated by these findings, we also investigate whether the recently proposed convexification approach of \cite{harrachCalderonProblemFinitely2023} can still lead to improved reconstructions. We propose a way to implement it in practice and conduct an extensive numerical comparison with the least squares approach, showing that the former is consistently outperformed by the latter.

\subsection{Previous works}
\label{sec_previous}
\paragraph{Theoretical study of the Calder\'on problem.} The landmark results of \cite{sylvesterGlobalUniquenessTheorem1987,nachmanGlobalUniquenessTwoDimensional1996} (see also the later works \cite{astala2006calderon,bukhgeimRecoveringPotentialCauchy2008,caroGlobalUniquenessCalderon2016}) show that, under suitable assumptions, the unknown conductivity $\gamma$ is entirely determined by the Neumann-to-Dirichlet map $\Lambda(\gamma)$. We stress that this amounts to assume that one has access to infinitely many measurements. More recently, identifiability from finitely many measurements has been investigated in \cite{albertiCalderonsInverseProblem2019,albertiInfiniteDimensionalInverseProblems2022}. Regarding robustness to measurement noise, the ill-posedness of the problem only allows for weak stability results \citep{alessandriniStableDeterminationConductivity1988,mandacheExponentialInstabilityInverse2001}, unless some strong a priori information on the unknown is available \citep{alessandriniLipschitzStabilityInverse2005,bacchelliLipschitzStabilityStationary2006,berettaLipschitzStabilityInverse2013,berettaGlobalLipschitzStability2022}.

\paragraph{Reconstruction methods.} The most common approach to solve the Calder\'on problem is to rely on Landweber iteration, which amounts to performing gradient descent on a nonlinear least squares objective. As this objective is non-convex, convergence to a global minimizer is not guaranteed. Many works have been devoted to the introduction of conditions under which local convergence can be proved (see for example \cite{hankeConvergenceAnalysisLandweber1995,neubauerLandweberIterationNonlinear2000,kaltenbacher-neubauer-scherzer-2008}). Whether these conditions are satisfied in the context of the Calder\'on problem is, to our knowledge, an open question (see \cite{Lechleiter-Rieder-2008,kindermann-2022,kaltenbacher-2024} for some partial results). In this setting, global convergence seems very difficult to prove. Another class of reconstruction techniques are direct methods \cite{nachmanGlobalUniquenessTwoDimensional1996,siltanenImplementationReconstructionAlgorithm2000}, which however are not considered in this work.

\paragraph{Convex programming approaches.} In \cite{harrachCalderonProblemFinitely2023}, it is proved that, when the number of measurements is sufficiently large, the sought-after conductivity is the unique solution of a convex program which depends on the unknown only through the measurements. This property is particularly interesting, since convex optimization problems can be solved with globally convergent algorithms (that is to say, whose convergence is guaranteed for every choice of initial point). To our knowledge, this is the first convexification result that is valid in the finite measurements and noisy setting. However, the proof of this result is not constructive, and no procedure to find the minimal number of measurements is proposed. In addition, the property only holds under strong conditions on the parameters defining the convex program. How to find admissible values for these parameters in practice is not discussed. As a result, this reconstruction method has not been implemented and numerically compared with existing approaches. Let us also mention that other convexification methods based on Carleman estimates have been proposed in \cite{klibanov2017globally,klibanovConvexificationElectricalImpedance2019} in the case of infinitely many measurements and diminishing noise. Finally, a convex approach based on lifting is discussed in \cite{albertiConvexLiftingApproach2025}.

\paragraph{Radial conductivities.} Numerous works were dedicated to the special setting of radial conductivities (see for instance \cite{sylvesterConvergentLayerStripping1992,knudsenDBarMethodElectrical2007,barceloBornApproximationThreedimensional2024,daudeStableFactorizationCalderon2024}). In \cite{siltanenImplementationReconstructionAlgorithm2000}, the explicit expression of the forward map for radial \emph{and} piecewise constant conductivities is computed. Assuming infinitely measurements are available, \cite{garde2020,garde-2022} develops a reconstruction algorithm based on a series of one-dimensional optimization problems for a particular class of piecewise constant conductivities, which covers the radial setting. To our knowledge, the only work studying the optimization landscape of least squares in this setting is \citet[Section 2]{harrachCalderonProblemFinitely2023}.

\subsection{Contributions}
We focus on the simple setting of piecewise constant radial conductivities, which allows us to obtain a finite-dimensional inverse problem. We prove two new properties in the case of conductivities that are defined by two scalar unknowns. We show that they are uniquely determined by two scalar measurements, and that the corresponding least squares objective in the noiseless case has no spurious critical points. We also prove that these results can be extended to the higher dimensional case under a numerically verifiable assumption. We visualize these facts in the simpler setting with a one-dimensional unknown and one scalar measurement in Figure~\ref{fig_1d}: despite the non-convexity, the squared loss has a single critical point, corresponding to the true unknown, and there are no local minima.

\begin{figure}
    \centering
    \subfloat[]{\label{fig_obj_1d}\includegraphics[width=3in]{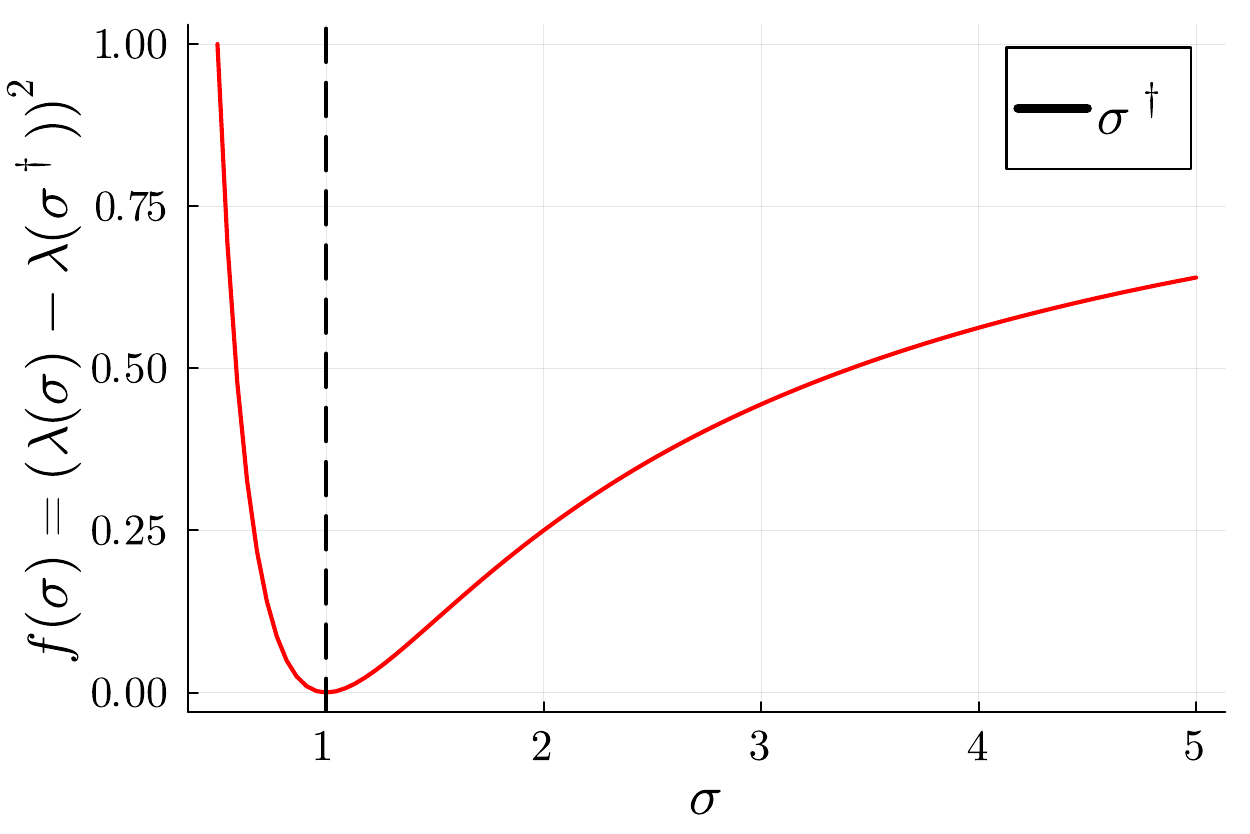}}~
    \subfloat[]{\label{fig_d2_obj_1d}\includegraphics[width=3in]{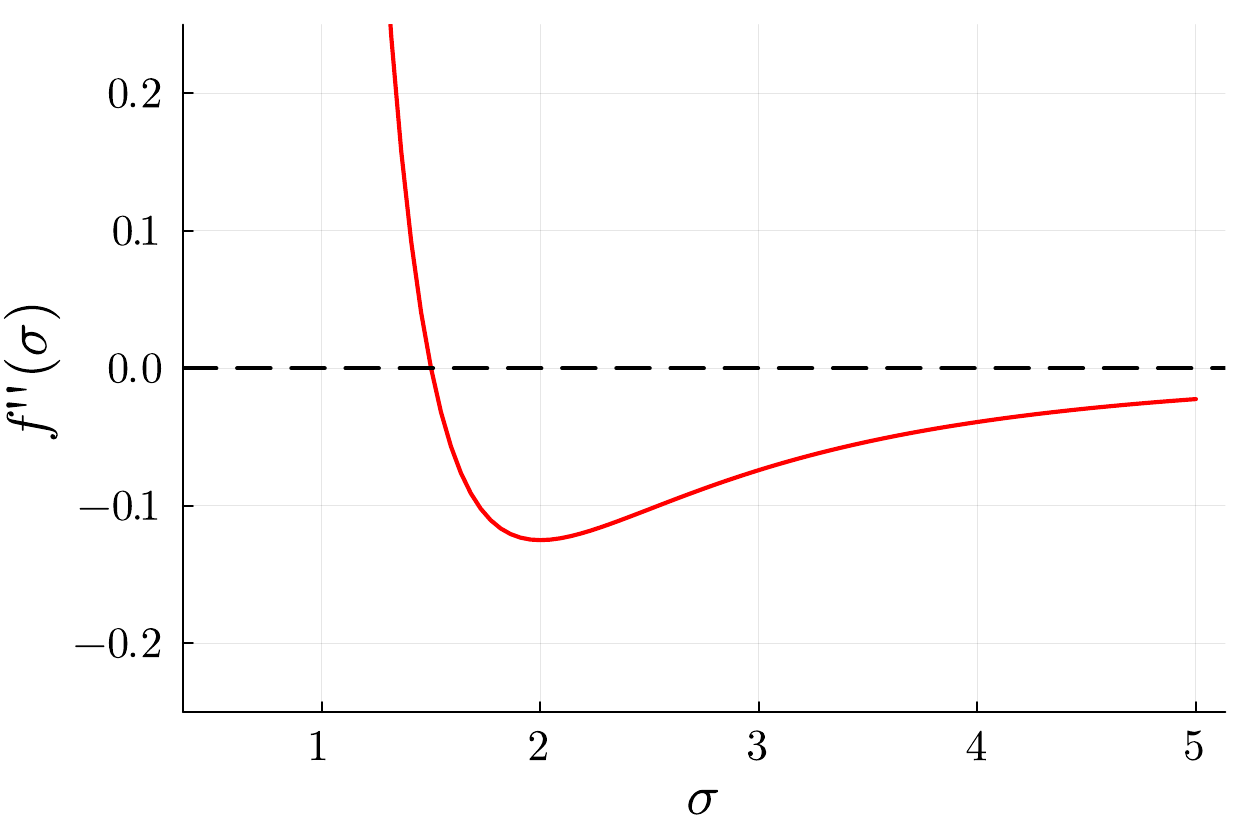}}
    \caption{plot of the least squares objective (a) and its second derivative (b) in a one-dimensional setting. The objective is not convex (its second derivative is not nonnegative), but it has a single critical point, which is its global minimum.}
    \label{fig_1d}
\end{figure}

Concerning the convex reformulation of the Cader\'on problem introduced in \cite{harrachCalderonProblemFinitely2023}, we propose a numerical procedure to evaluate this formulation. We first introduce a method to estimate the parameters of the convex program by exploiting its optimality conditions.  We then show that the convex program can be solved efficiently with interior point algorithms. We extensively compare this approach with the classical least squares approach. We provide an open source Julia package called \texttt{RadialCalderon.jl}\footnote{\href{https://github.com/rpetit/RadialCalderon.jl}{https://github.com/rpetit/RadialCalderon.jl} (see also \href{https://rpetit.github.io/RadialCalderon.jl}{https://rpetit.github.io/RadialCalderon.jl} for tutorials and the full documentation).} which allows one to reproduce all our experiments (see \cite{bezansonJuliaFreshApproach2017} for more background on the Julia language). This package can be used to benchmark different reconstruction methods and study numerically this radial piecewise constant version of the Calder\'on problem.

Our results suggest that the convex programming approach only allows for an accurate estimation of the unknown for problems with a very small size. It is consistently outperformed by the least squares approach, which is faster and more accurate while requiring less measurements. We argue that circumventing the nonconvexity of the problem without addressing its ill-posedness cannot yield an accurate estimation procedure.

\paragraph{Plan of the paper.} In \Cref{sec_pc_radial}, we introduce the setting of piecewise constant radial conductivities. In \Cref{sec_theory}, we analyze the case of conductivities that are defined by two scalar unknowns, and discuss the extension of our results to the higher dimensional setting. In \Cref{sec_convex}, we introduce the resolution method proposed in \cite{harrachCalderonProblemFinitely2023} and propose a method for estimating all the parameters of the related convex program. In \Cref{sec_num_comparison}, we perform an extensive numerical comparison of this method with the least squares approach. Finally, some concluding remarks are discussed in Section~\ref{sec:conclusion}.

\section{Piecewise constant radial conductivities}
\label{sec_pc_radial}
In this section, we introduce the setting of radial piecewise constant conductivities, which has all the important features of the full Calder\'on problem. Its main interest is to allow the evaluation of the forward map without having to rely on finite element methods to solve the underlying PDE. This allows us to decouple the study of the inverse problem and its ill-posedness from the effect of numerical errors in the evaluation of the forward map.

\subsection{Setting}\label{sub:setting}
We work on the two-dimensional open unit ball $\Omega=B(0,1)$ and assume that the unknown conductivity is piecewise constant on a known partition of $\Omega$ made of concentric annuli. More precisely, we use the polar coordinates $(r,\theta)$ and assume that $\gamma(r,\theta)$ is equal to a positive constant $\sigma_i>0$ for every $r_{i}<r<r_{i-1}$ with $0=r_n<r_{n-1}<...<r_0=1$ (see \Cref{fig_radial} for a schematic description). For every vector $\sigma\in(\RR_{>0})^n$, we denote by $\gamma_{\sigma}$ the associated function, where $\RR_{>0}$ denotes the set of positive real numbers.

\begin{figure}
    \centering
    \input{plots/fig_radial}
    \caption{schematic description of the piecewise constant radial setting when $n=4$. The conductivity $\gamma:B(0,1)\to \RR$ is such that, using the polar coordinates $(r,\theta)$, it holds $\gamma(r,\theta)=\sigma_i$ for every $r_{i}<r<r_{i-1}$ for some $\sigma=(\sigma_i)_{1\leq i\leq n}\in (\RR_{>0})^n$.}
    \label{fig_radial}
\end{figure}
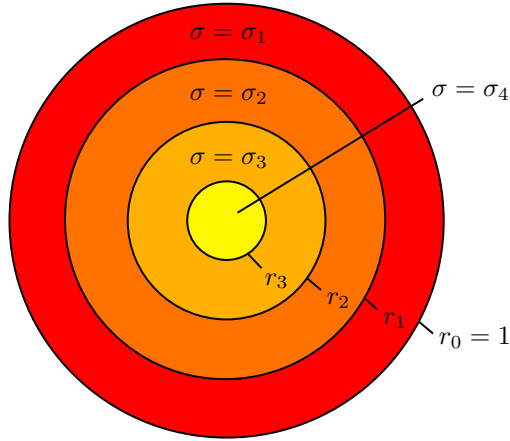

As noticed in \citet[Lemma 4.1]{siltanenImplementationReconstructionAlgorithm2000}, taking the Neumann data to be cosine functions allows us to obtain a diagonal Neumann-to-Dirichlet map. To be more precise, we fix a number of scalar measurements $m$, and,  given $j\in\{1,...,m\}$, we define $g_j:\partial \Omega\to\RR$ by $g_j(\theta)=\cos(j\theta)$. We notice as in \citet[Lemma 4.1]{siltanenImplementationReconstructionAlgorithm2000} that if $u_j(r,\theta)=(\alpha_{i,j}r^j+\beta_{i,j}r^{-j})\mathrm{cos}(j\theta)$ for every $r_{i}<r<r_{i-1}$ then $u_j$ is a solution to the Neumann problem for the conductivity equation
\begin{equation*}
    \begin{cases}
        \mathrm{div}(\gamma_\sigma\nabla u_j)=0 &\mathrm{in}~\Omega,\\
        \gamma_\sigma \partial_{\nu}u_j=g_j &\mathrm{on}~\partial\Omega,
    \end{cases}
\end{equation*}
provided $\beta_{n,j}=0$, $\sigma_1j(\alpha_{1,j}-\beta_{1,j})=1$ and
\begin{equation}
    \left\{
    \begin{aligned}
        \alpha_{i,j}r_i^j+\beta_{i,j}r_i^{-j}&=\alpha_{i+1,j}r_i^j+\beta_{i+1,j}r_i^{-j},\\
        \sigma_i(\alpha_{i,j}r_i^j-\beta_{i,j}r_i^{-j})&=\sigma_{i+1}(\alpha_{i+1,j}r_i^j-\beta_{i+1,j}r_i^{-j}),
    \end{aligned}
    \right.
    \label{linear_system_pde_coeff}
\end{equation}
for every $1\leq i\leq n-1$ and $1\leq j\leq m$. It is shown in \citet[Lemma 4.1]{siltanenImplementationReconstructionAlgorithm2000} that the solution of \eqref{linear_system_pde_coeff} satisfies $\beta_{i,j}=C_{i,j}\alpha_{i,j}$ with
\begin{equation}
    \begin{aligned}
        C_{n,j}=0~~\mathrm{and}~~C_{i,j}=\frac{\rho_i r_i^{2j}+C_{i+1,j}}{1+\rho_i C_{i+1,j}r_i^{-2j}},
    \end{aligned}
    \label{rec_rel_C}
\end{equation}
where $\rho_i\eqdef(\sigma_{i}-\sigma_{i+1})/(\sigma_{i}+\sigma_{i+1})$.

Since $\restriction{u_j}{\partial\Omega}=(\alpha_{1,j}+\beta_{1,j})\cos(j\theta)$, we obtain that the Neumann-to-Dirichlet map introduced in \Cref{subsec_eit_calderon} satisfies $\Lambda(\gamma_{\sigma})g_j=\lambda_j(\sigma)g_j$ with 
\begin{equation}\label{eq_lambda_j}
    \lambda_j(\sigma)\eqdef \frac{1+C_{1,j}}{j\sigma_1(1-C_{1,j})}.
\end{equation}
\paragraph{Finite-dimensional forward map.} Abusing notation, we also denote by $\Lambda$ the following truncated Neumann-to-Dirichlet map:
\begin{equation*}
    \begin{aligned}
        \Lambda \colon (\RR_{>0})^n &\to \RR^m\\
        \sigma &\mapsto [\lambda_j(\sigma)]_{j=1}^m.
    \end{aligned}
\end{equation*}
In the following, we focus on the recovery of $\sigma$ from the knowledge of $\Lambda(\sigma)$, which is an inverse problem with a finite-dimensional unknown and finitely many measurements.

In practice, the forward map $\Lambda$ can be evaluated by computing $C_{1,j}$ using \eqref{rec_rel_C} and plugging the result in the expression of $\lambda_j(\sigma)$. The Jacobian and the Hessian of $\Lambda$ can also be computed efficiently via automatic differentiation. To compute them, we use the package \texttt{ForwarDiff.jl} \citep{revelsForwardModeAutomaticDifferentiation2016} via \texttt{DifferentiationInterface.jl} \citep{dalleDifferentiationInterface2025,hill2025sparserbetterfasterstronger,schäfer2022abstractdifferentiationjlbackendagnosticdifferentiableprogramming}.

The following proposition gives some properties of the forward map $\Lambda$, on which we crucially rely.
\begin{proposition}
    The mapping $\Lambda$ is convex. In addition, it is strictly decreasing, meaning that for every $\sigma,\sigma'\in (\RR_{>0})^n$ with $\sigma\leq \sigma'$ (pointwise) and $\sigma\neq \sigma'$, it holds $\lambda_j(\sigma)>\lambda_j(\sigma')$ for every $1\leq j\leq m$.
    \label{prop_forward_map}
\end{proposition}
\begin{proof}
    The convexity of $\Lambda$ is proved in \citet[Lemma 4.1]{harrachCalderonProblemFinitely2023}. From the same reference, we have that the partial derivative of $\lambda_j$ with respect to its $i$-th variable (that is $\sigma_i$) satisfies
    \begin{equation}\label{eq:derivative_explicit}
        \partial_{i}\lambda_j(\sigma)=-\frac{1}{\pi}\int_{A_i}\|\nabla u_j(\tau)\|_2^2 d\tau,
    \end{equation}
    where $A_i$ denotes the annulus $r_i<r<r_{i-1}$. We claim that this quantity is negative for every $1\leq i\leq n$ and $1\leq j \leq m$. To see this, we notice that, since $\gamma_{\sigma}$ is lower bounded by the positive constant $\mathrm{min}_{1\leq i\leq n}\,\sigma_i$, the differential operator $\mathrm{div}(\gamma_{\sigma}\nabla\cdot)$ is uniformly elliptic. As a result, the strong unique continuation principle (see for instance \cite{alessandriniStrongUniqueContinuation2012}) ensures that, if $u_j$ vanishes on a non empty open set, it is identically zero. Since the boundary data $g_j$ is not identically zero, we obtain that $u_j$ is not identically zero. If $u_j$ was equal to a constant on a non empty open set, the same argument could be applied to $u_j$ minus this constant, which still solves $\mathrm{div}(\gamma_{\sigma}\nabla u_j)=0$. As a result, we obtain that $\nabla u_j$ cannot vanish on a non empty open set, which yields the result.
\end{proof}

\subsection{Numerical study of the inverse problem}
\label{subsec_numerical_study_ip}
In this section, we present several experiments suggesting that this simple piecewise constant radial model exhibits all the important features of the classical Calder\'on problem. In particular, the estimation problem becomes increasingly ill-posed away from the boundary. In the following, we take $r_i=(n-i)/n$ ($i=0,...,n$).

Our first experiment shows that the estimation of the conductivity is much harder on annuli that are distant from the boundary. We take $n=m=10$ and draw $100$ random true conductivities uniformly in $[a,b]^n$ with $a=0.5$ and $b=1.5$. Then, we use a trust-region algorithm to solve the nonlinear equation $\Lambda(\sigma)=\Lambda(\sigma^\dagger)$ for each true conductivity $\sigma^\dagger$. For each instance, we require that the estimated conductivity $\hat{\sigma}$ satisfies $\|\Lambda(\hat{\sigma})-\Lambda(\sigma^\dagger)\|_{\infty}\leq 10^{-15}$, where $\|\cdot\|_{\infty}$ denotes the euclidean $\ell^{\infty}$ norm (the maximum of the absolute values of the coordinates). Then, we compute the mean of $|\hat{\sigma}_i-\sigma^\dagger_i|$ for each $i\in\{1,...,n\}$. The result is given in \Cref{fig_error_outer_inner}. We notice that the average error increases by several orders of magnitude as $i$ grows. This is consistent with the well-known fact that the reconstruction quality deteriorates as one moves away from the boundary, see e.g.\ \cite{winkler-rieder-2014,garde-knudsen-2017,alessandrini-scapin-2017,garde-nyvonen-2020}. Taking $m$ much larger than $n$ yields similar results. 

Our second experiment shows that the inverse problem becomes more and more ill-posed as $n$ grows, as expected given the severe ill-posedness of the infinite-dimensional inverse problem, see \cite{rondi-2006}. Indeed, given $\epsilon>0$, we notice that the quantity
\begin{equation*}
    \underset{\sigma,\sigma'\in[a,b]^n}{\mathrm{sup}}~\|\sigma-\sigma'\|_{\infty}~~\mathrm{s.t.}~~\|\Lambda(\sigma)-\Lambda(\sigma')\|_{\infty}\leq\epsilon
\end{equation*}
significantly increases as $n$ increases. To show this, we run the same experiment as above with $n=m$ ranging from $2$ to $10$. The results are displayed in \Cref{fig_mean_error_npixels}. Again, taking $m$ much larger than $n$ yields similar results.
\begin{figure}
    \centering
    \includegraphics[width=0.5\linewidth]{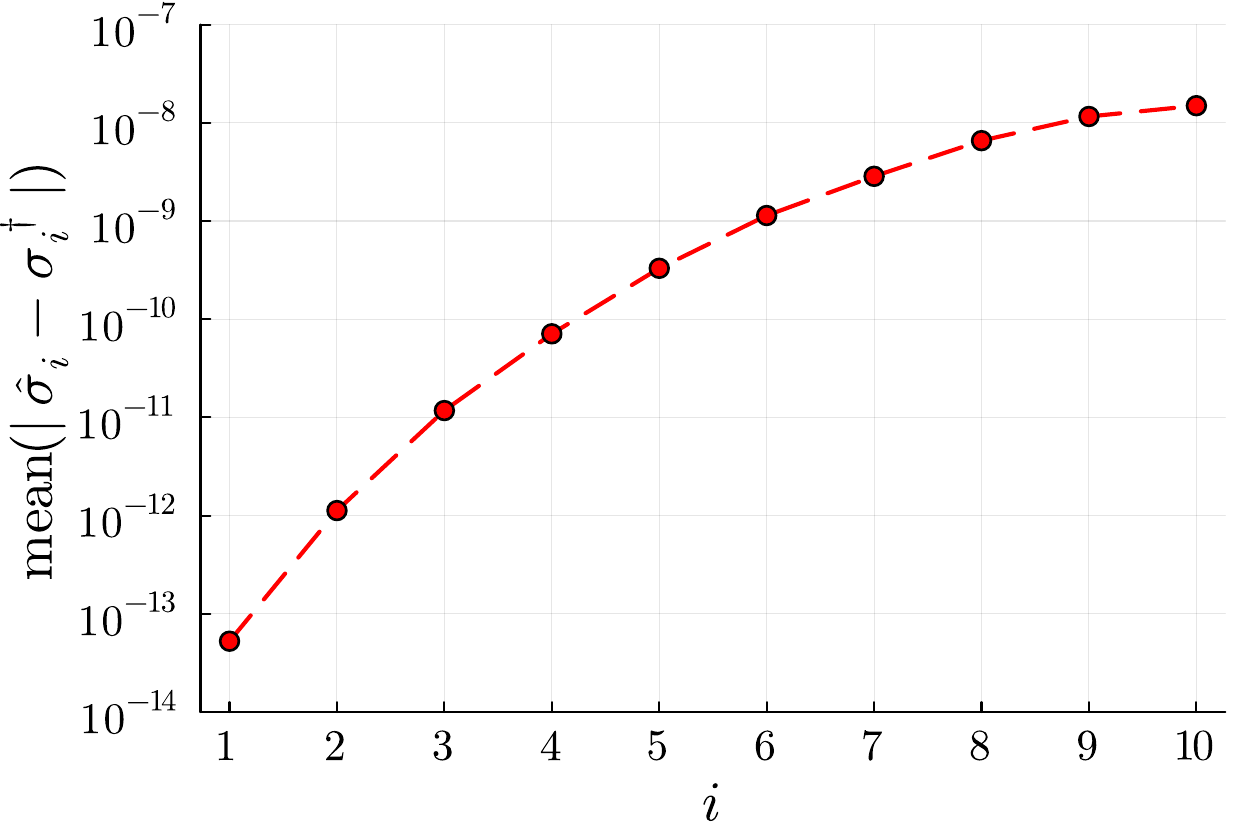}
    \caption{mean error on the $i$-th annulus among $100$ pairs $\sigma^\dagger,\hat{\sigma}\in[a,b]^n$ such that $\|\Lambda(\hat{\sigma})-\Lambda(\sigma^\dagger)\|_{\infty}\leq 10^{-15}$ with $(a,b)=(0.5, 1.5)$ and $n=m=10$. The error on the outermost annulus is several orders of magnitude smaller than the error on the innermost annulus.}
    \label{fig_error_outer_inner}
\end{figure}

\begin{figure}
    \centering
    \includegraphics[width=0.5\linewidth]{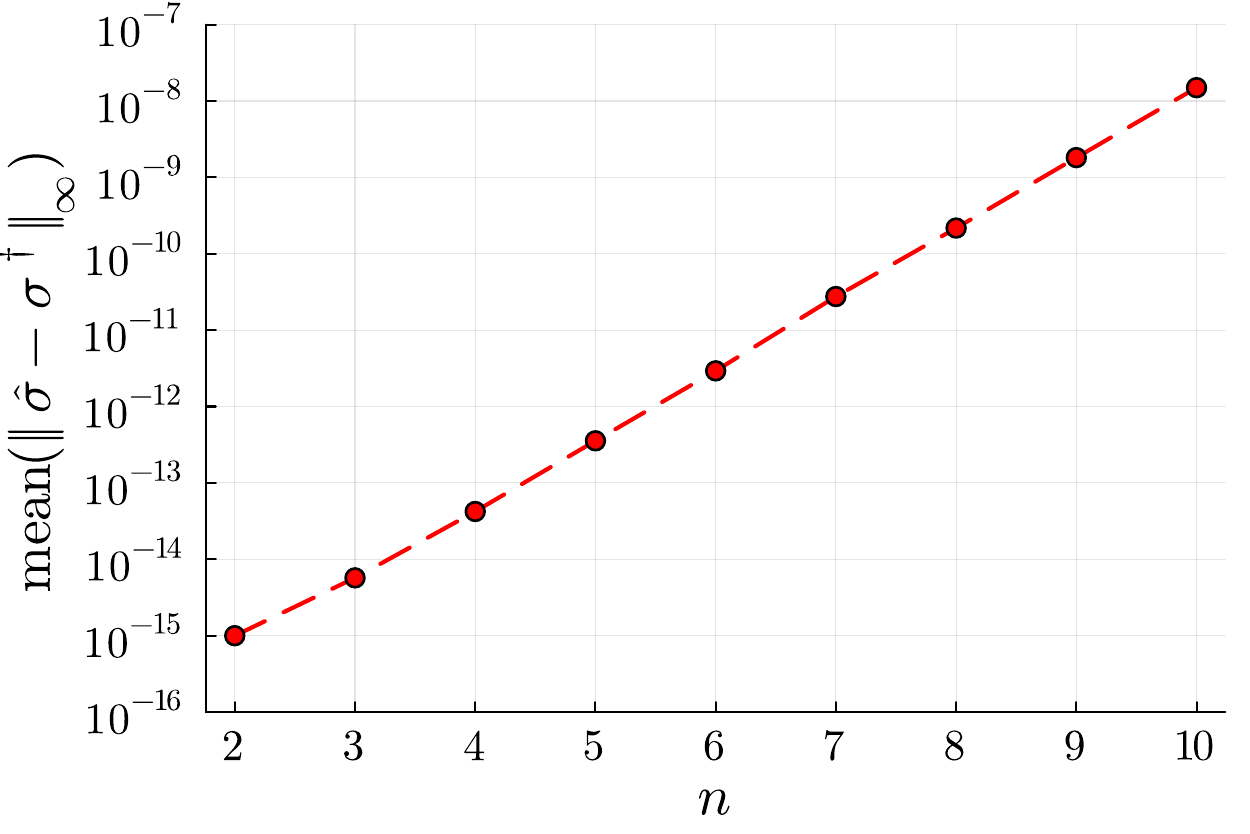}
    \caption{mean $\ell^\infty$ error as a function of the number of annuli $n$. The mean is computed on $100$ pairs $\sigma^\dagger,\hat{\sigma}\in[a,b]^n$ such that $\|\Lambda(\hat{\sigma})-\Lambda(\sigma^\dagger)\|_{\infty}\leq 10^{-15}$ with $(a,b)=(0.5,1.5)$ and $m=n$ measurements.}
    \label{fig_mean_error_npixels}
\end{figure}

Finally, we notice that $\Lambda'(\sigma)$ (the Jacobian matrix of $\Lambda$ at $\sigma$) seems to be always invertible when $m=n$, but that its smallest singular value goes to $0$ as $n$ increases (see \Cref{table_jac_det} in \Cref{sec_theory}). When $m>n$, there are always $n$ nonzero singular values.

\section{Identifiability and absence of bad critical points}
\label{sec_theory}
We prove in this section that, when $n=2$ and $m\geq 2$, the forward map $\Lambda$ is injective. We also show that, when $n=2$ and $m\geq n$, for every $\sigma^\dagger\in(\RR_{>0})^n$, the least squares objective $\sigma\mapsto\|\Lambda(\sigma)-\Lambda(\sigma^\dagger)\|_2^2$ has a unique critical point, which is $\sigma^\dagger$. Finally, we provide a discussion concerning the case where $n>2$. Our results suggest that, in this piecewise constant radial setting, the main difficulty for solving the inverse problem is not the nonconvexity of the least squares objective, but rather the ill-posedness of the problem.

Throughout, given $\Lambda: \sigma\in (\RR_{>0})^n \to (\lambda_i(\sigma))_{i=1}^m$, its Jacobian is denoted by $\Lambda'(\sigma)\in\RR^{m\times n}$ with its $(j,i)$-th entry in row $j$ and column $i$ being $\partial_i \lambda_j(\sigma)$.

\subsection{The two-dimensional case}
In the particular case of two scalar unknowns ($n=2$), \eqref{rec_rel_C} and \eqref{eq_lambda_j} yield \begin{equation}\label{eq_lambda_j_n_2}
    \lambda_j(\sigma)=\frac{1+\rho_1 r_1^{2j}}{j\sigma_1(1-\rho_1r_1^{2j})}
\end{equation}
for every $1\leq j\leq m$.

We first state the following useful lemma.
\begin{lemma}\label{lem:2d}
    Let $n=m=2$ and consider $\Lambda:\sigma\mapsto(\lambda_1(\sigma),\lambda_2(\sigma))$. Then, the determinant of the Jacobian is negative, that is to say $\det(\Lambda'(\sigma))<0$ for all $\sigma\in(\RR_{>0})^2$.
\end{lemma}
\begin{proof}
As shown in the proof of Proposition~\ref{prop_forward_map}, we have $\partial_i\lambda_j(\sigma)<0$. Let 
$$
h_j(\sigma) \eqdef \frac{\partial_1 \lambda_j (\sigma)}{\partial_2 \lambda_j(\sigma)}.
$$
Then, one can check (we obtained this expression by relying on a symbolic computation tool) that for every $\sigma=(\sigma_1,\sigma_2)\in(\RR_{>0})^2$, it holds  
\begin{align}
    h_j(\sigma) = \frac{1}{4\sigma_1^2} \left(\frac{1}{r_1^{2j}} (\sigma_1+\sigma_2)^2 - r_1^{2j}(\sigma_1-\sigma_2)^2 - 4\sigma_1\sigma_2\right) ,
\end{align}
which is clearly increasing in $j$, since $r_1\in (0,1)$ implies that each of the terms is increasing in $j$.
It follows that $$\det(\Lambda'(\sigma)) = \partial_1\lambda_1(\sigma) \partial_2\lambda_2(\sigma) - \partial_1\lambda_2(\sigma) \partial_2\lambda_1(\sigma) =\partial_2\lambda_2(\sigma)\partial_2\lambda_1(\sigma)[h_1(\sigma)-h_2(\sigma)]<0,$$
where we have used also that $\partial_i\lambda_1(\sigma)<0$ ($i=1,2$).
\end{proof}

Now, we prove that the unknown conductivity can be recovered from two scalar measurements. To our knowledge, the best identifiability results that are currently available require a number of measurements which is exponential in the number of unknowns \citep{albertiInfiniteDimensionalInverseProblems2022}.
\begin{proposition}\label{prop:2d_inj}
    If $n=2$ and $m\geq 2$, the forward map $\Lambda$ is injective.
\end{proposition}
\begin{proof}
   It is enough to prove the result in the case $m=2$. Let $\sigma^\dagger,\hat \sigma\in(\RR_{>0})^2$ be such that 
     $\sigma^\dagger \neq  \hat \sigma$ and $\Lambda(\sigma^\dagger) = \Lambda(\hat \sigma)$. By strict monotonicity {(Proposition~\ref{prop_forward_map})}, we can assume without loss of generality that $\sigma^\dagger_1<\hat \sigma_1$ and $\sigma^\dagger_2> \hat \sigma_2$.

     We first show that there exists a function $g:[\sigma^\dagger_1,\hat \sigma_1] \to [ \hat \sigma_2,\sigma^\dagger_2]$ such that $\lambda_1(\sigma_1,g(\sigma_1)) = \lambda_1(\sigma^\dagger)$ for every $\sigma_1\in[\sigma_1^\dagger,\hat{\sigma}_1]$. For visualization purposes, the graph of this function, as well as the graph of the function $h$ that we compute below, is given in \Cref{fig_g_h} in the case where $\sigma^\dagger=\mathbf{1}$. Define $f(\sigma) = \lambda_1(\sigma) - \lambda_1(\sigma^\dagger)$. Then, by monotonicity,
     \begin{itemize}
         \item $f$ is negative on $(\sigma^\dagger_1,\hat \sigma_1)\times \{\sigma_2^\dagger\}$,
         \item $f$ is positive on $(\sigma^\dagger_1,\hat \sigma_1)\times \{\hat \sigma_2\}$.
     \end{itemize}
     It follows by the intermediate value theorem that for all $\sigma_1 \in (\sigma^\dagger_1,\hat \sigma_1)$, there exists $\sigma_2 \in (\hat \sigma_2,\sigma^\dagger_2)$ such that $f(\sigma_1,\sigma_2) = 0$. This shows that $g$ is well defined. Moreover, since $\partial_2 \lambda_1 <0$, it follows by the implicit function theorem that $g$ is of class $C^1$ and $g'(\sigma_1) = - \partial_1 \lambda_1(\sigma, g(\sigma_1))/\partial_2 \lambda_1(\sigma_1, g(\sigma_1)).$

 \begin{figure}
    \centering
    \subfloat[]{\includegraphics[width=3in]{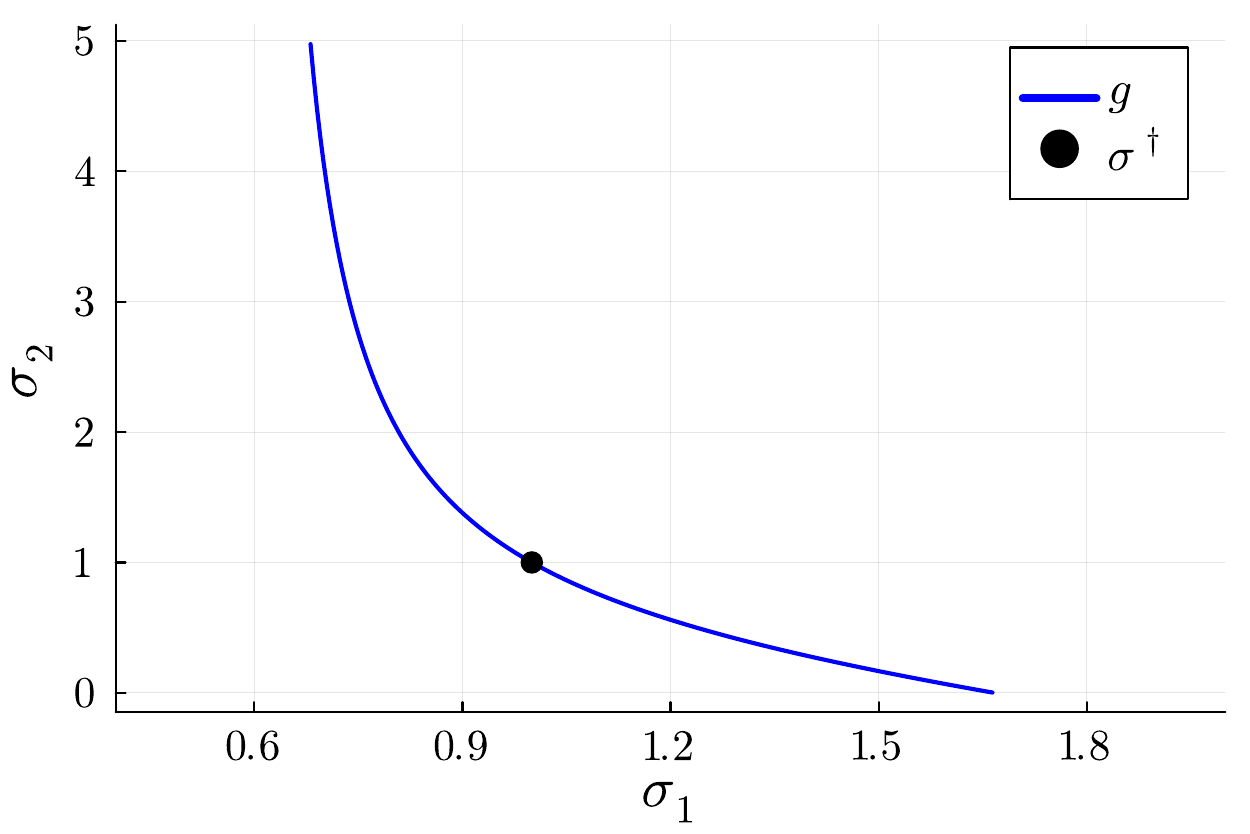}}~
    \subfloat[]{\includegraphics[width=3in]{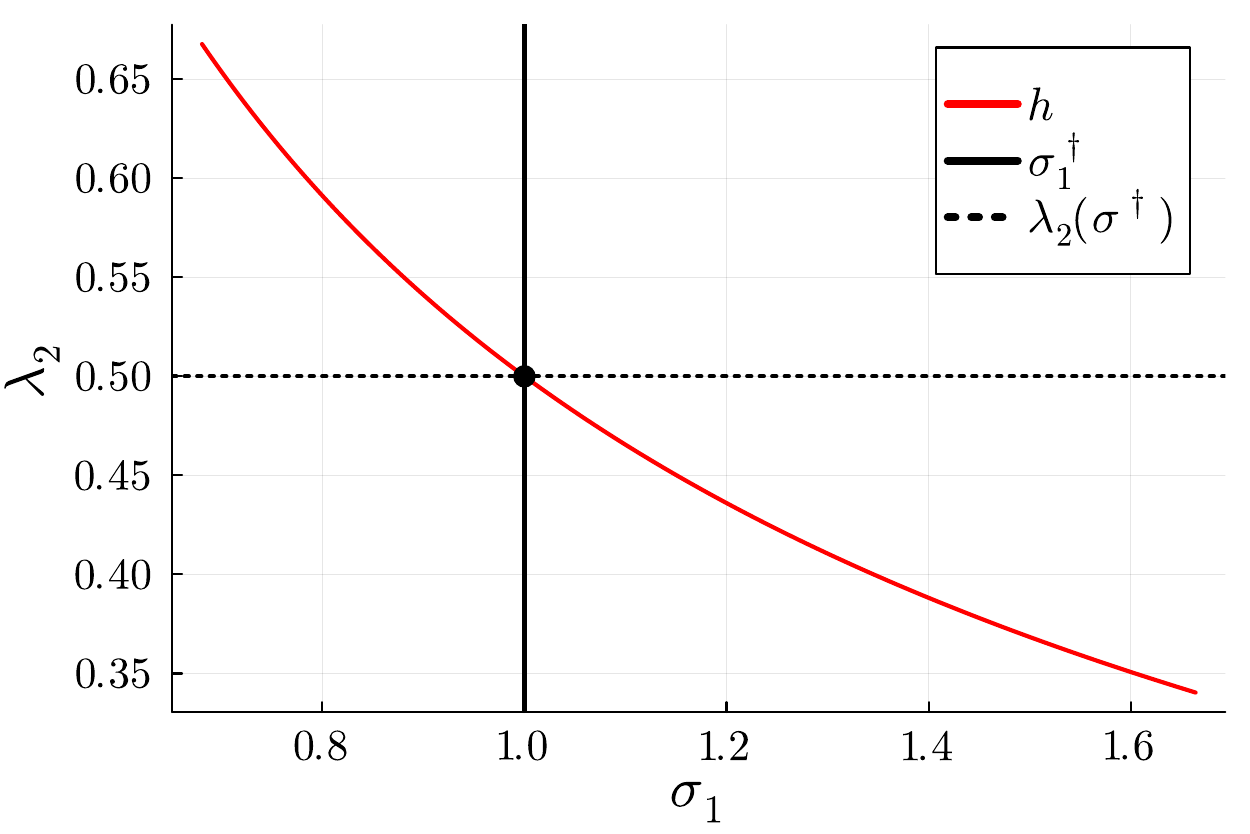}}
\caption{graph of the functions $g$ and $h$ constructed in the proof of \Cref{prop:2d_inj}, with $\sigma^\dagger=\mathbf{1}$. The function $h$ is strictly monotone, showing that the forward map is injective.}
    \label{fig_g_h}
\end{figure}

     Now consider $h(\sigma_1) = \lambda_2(\sigma_1,g(\sigma_1))$. Since $g(\sigma_1^\dagger) = \sigma_2^\dagger$ and $g(\hat \sigma_1) =\hat  \sigma_2$, by  assumption, $h(\sigma_1^\dagger) = h(\hat \sigma_1)$. Observe also that
     $$
     h'(\sigma_1) = \partial_1\lambda_2(\sigma_1,g(\sigma_1)) + \partial_2 \lambda_2(\sigma_1,g(\sigma_1)) g'(\sigma_1) = \frac{(\partial_1\lambda_2 \partial_2 \lambda_1 - \partial_1\lambda_1 \partial_2\lambda_2) }{\partial_2\lambda_1},
     $$
     where each term on the right-hand side is evaluated at $(\sigma_1,g(\sigma_1))$. We have $h'(\sigma_1)>0$ for every $\sigma_1 \in (\sigma^\dagger_1,\hat \sigma_1)$ by Lemma~\ref{lem:2d}, which contradicts the assumption that $h$ takes the same value on two distinct points.    
\end{proof}

Finally, we prove that, when $n=2$ and $m\geq 2$, the least squares objective does not have any spurious critical point.
\begin{proposition}
    If $n=2$ and $m\geq 2$, for every $\sigma^\dagger\in(\RR_{>0})^n$, the least squares objective $$f:\sigma\mapsto \frac{1}{2}\|\Lambda(\sigma)-\Lambda(\sigma^\dagger)\|_2^2$$ has a unique critical point on $(\RR_{>0})^n$, which is $\sigma^\dagger$.
    \label{prop_critical_point}
\end{proposition}

\begin{proof}
    First, let us stress that the result for $m=2$ follows directly from \Cref{lem:2d}. Indeed, if $\sigma$ is a critical point then $\nabla f(\sigma)=[\Lambda'(\sigma)]^T(\Lambda(\sigma) - \Lambda(\sigma^\dagger)) = 0$. Since $[\Lambda'(\sigma)]^T\in\RR^{2\times 2}$ is injective (by \Cref{lem:2d}), this implies $\Lambda(\sigma) - \Lambda(\sigma^\dagger)= 0$. By injectivity of $\Lambda$ (Proposition~\ref{prop:2d_inj}), we obtain $\sigma= \sigma^\dagger$. This argument does not carry on to the general case, as $[\Lambda'(\sigma)]^T\in \RR^{2\times m}$ cannot be full rank when $m>2$.

    For the general case $(m\geq 2)$, we define
    \begin{equation*}
        \begin{aligned}
            \Phi\colon (\RR_{>0})^2&\to \RR_{>0}\times (-1,1)\\
            (\sigma_1,\sigma_2)&\mapsto \bigg(\frac{1}{\sigma_1},\frac{\sigma_1-\sigma_2}{\sigma_1+\sigma_2}\bigg)
        \end{aligned}
        \qquad\qquad\mathrm{and}\qquad\qquad
        \begin{aligned}
            \mu_j\colon \RR_{>0}\times (-1,1)&\to \RR\\
            (x, y)&\mapsto \frac{x}{j}\frac{1+r_1^{2j}y}{1-r_1^{2j}y}
        \end{aligned}
    \end{equation*}
    as well as $(x^\dagger,y^\dagger)=\Phi(\sigma^\dagger)$ and $g:(x,y)\in\RR_{>0}\times (-1,1)\to (1/2)\sum_{j=1}^m (\mu_j(x,y)-\mu_j(x^\dagger,y^\dagger))^2$. As $f=g\circ \Phi$ and $\Phi$ is a diffeomorphism, a point $\sigma$ is a critical point of $f$ if and only if $\Phi(\sigma)$ is a critical point of $g$. Therefore, we only have to prove that $(x^\dagger, y^\dagger)$ is the only critical point of $g$.

    Let $(x,y)\in \RR_{>0}\times (-1,1)$ be fixed. We will show that, if $(x,y)\neq (x^\dagger,y^\dagger)$, then $(x,y)$ is not a critical point of $g$. Let us first assume that 
    \begin{equation}\label{ass_sign}
        (x-x^\dagger)\bigg(x\frac{1+y}{1-y}-x^\dagger \frac{1+y^\dagger}{1-y^\dagger}\bigg)<0.
    \end{equation}
    Let $R$ be the unique real number in $(0,1)$ such that
    \begin{equation*}
        x \frac{1+Ry}{1-Ry}=x^\dagger \frac{1+Ry^\dagger}{1-R y^\dagger}
    \end{equation*}
    which exists and is unique by \eqref{ass_sign}. Then, for every $r\in\RR$, it holds (from the factorization rules for degree two polynomials)
    \begin{equation*}
        x\frac{1+ry}{1-ry}-x^\dagger \frac{1+ry^\dagger}{1-ry^\dagger}=\frac{x^\dagger-x}{(1-ry)(1-ry^\dagger)}(r-R)(yy^\dagger r+1/R).
    \end{equation*}
    As a result, for every $1\leq j\leq m$, it holds 
    \begin{equation*}
        \mu_j(x,y)-\mu_j(x^\dagger,y^\dagger)=\frac{1}{j}\frac{x^\dagger-x}{(1-r_1^{2j}y)(1-r_1^{2j}y^\dagger)}(r_1^{2j}-R)(yy^\dagger r_1^{2j}+1/R).
    \end{equation*}
    For every $1\leq j\leq m$ and $(\alpha,\beta)\in\RR^2$, it holds
    \begin{equation*}
        d\mu_j(x,y)\cdot (\alpha,\beta)=\frac{1}{j(1-r_1^{2j}y)}\bigg[\alpha (1+r_1^{2j}y)+\frac{2r_1^{2j}\beta x}{1-r_1^{2j}y}\bigg].
    \end{equation*}
    We evaluate this expression at $\alpha=-2Rx/(1-Ry)$ and $\beta=1+Ry$ and obtain
    \begin{equation*}
        \begin{aligned}
            d\mu_j(x,y)\cdot(\alpha,\beta)&=\frac{1}{j(1-r_1^{2j}y)}\bigg[-\frac{2R(1+r_1^{2j}y)x}{1-Ry}+\frac{2r_1^{2j}(1+Ry)x}{1-r_1^{2j}y}\bigg]\\
            &=\frac{2x}{j(1-r_1^{2j}y)^2 (1-Ry)}(r_1^{2j}-R)(1+R r_1^{2j}y^2).
        \end{aligned}
    \end{equation*}
    Now, it holds that 
    \begin{equation*}
        \begin{aligned}
            dg(x,y)\cdot (\alpha,\beta)&=\sum\limits_{j=1}^m (\mu_j(x,y)-\mu_j(x^\dagger,y^\dagger))d\mu_j(x,y)\cdot(\alpha,\beta)\\
            &=\sum\limits_{j=1}^m \frac{1}{j^2}\frac{2x(x^\dagger-x)}{(1-r_1^{2j}y)^3(1-r_1^{2j}y^\dagger)(1-Ry)}(r_1^{2j}-R)^2 (yy^\dagger r_1^{2j}+1/R)(1+Rr_1^{2j}y^2).
        \end{aligned}
    \end{equation*}
    All terms inside the sum have the same sign as $x^\dagger-x$. Indeed $2x$, $1-r_1^{2j}y$, $1-r_1^{2j}y^\dagger$, $1-Ry$ and $1+Rr_2^{2j}y^2$ are all positive and $yy^\dagger r_1^{2j}+1/R\geq -|y||y^\dagger|r_1^{2j}+1/R> -1+1/R>0$ (since $R\in(0,1)$). The terms in the sum can only be zero if $r_1^{2j}=R$, which occurs for at most one value of $j$. Therefore, we have a sum of real numbers which have the same sign and are not all zero, meaning that $dg(x,y)\cdot(\alpha,\beta)\neq 0$ and $(x,y)$ is not a critical point.

    Now, we consider the case where \eqref{ass_sign} does not hold. This case is easier, as one can show that the sign of 
    \begin{equation*}
        r\in[0,1]\mapsto x\frac{1+ry}{1-ry}-x^\dagger \frac{1+ry^\dagger}{1-ry^\dagger}
    \end{equation*}
    does not change. Indeed, the fact that~\eqref{ass_sign} does not hold means that this map has the same sign at $0$ and $1$. If $yy^\dagger\leq 0$, it is a monotonic function, so it cannot change sign over $[0,1]$. If $yy^\dagger > 0$, one can check by an explicit computation that it has two zeros, with opposite signs, so that it cannot cancel more than once in $(0,1)$. If it has the same sign at $0$ and $1$, it therefore does not cancel in $(0,1)$ at all, and the sign does not change over the interval.
    
    Hence $\mu_j(x,y)-\mu_j(x^\dagger,y^\dagger)$ has the same sign for all values of $j$. Therefore, we have that 
    \begin{equation*}
        dg(x,y)\cdot(1,0)=\sum\limits_{j=1}^m (\mu_j(x,y)-\mu_j(x^\dagger,y^\dagger))\frac{1+r_1^{2j}y}{j(1-r_1^{2j}y)}
    \end{equation*}
    is a sum of terms with the same sign. Unless $\mu_j(x,y)=\mu_j(x^\dagger,y^\dagger)$ for every $1\leq j\leq m$ (which is possible only if $(x,y)=(x^\dagger,y^\dagger)$), not all terms are zero, hence the sum is also non-zero and $(x,y)$ is not a critical point.
\end{proof}

\subsection{The higher dimensional case}
In this section, we discuss the extension of the above results to the case $n>2$. We restrict ourselves to the case $m=n$. Given $\sigma\in(\RR_{>0})^n$, we denote by $\Lambda_{[n-1]}'(\sigma) \in \RR^{(n-1)\times (n-1)}$ the submatrix of $\Lambda'(\sigma)\in\RR^{n\times n}$ obtained by removing its last row (corresponding to the last measurement) and its first column (corresponding to the outermost annulus).

\subsubsection{Main result}
Our main result shows that, provided the following two conditions hold, the least squares objective has a unique critical point, which is the unknown conductivity. We discuss these conditions in greater details at the end of this section.
\begin{itemize}
    \item[(i)] Jacobian invertibility: for every $n\geq 2$ and every $\sigma\in (\RR_{>0})^n$, $\Lambda'(\sigma)$ and $\Lambda'_{[n-1]}(\sigma)$ are invertible.

    \item[(ii)] Alternating signs: for every $n\geq 2$, if $\sigma^\dagger, \hat \sigma \in (\RR_{>0})^n$ are such that $\sigma^\dagger_1 < \hat \sigma_1$ and $\lambda_j(\sigma^\dagger) = \lambda_j(\hat \sigma)$ for every $j=1,\ldots, n-1$, then $\mathrm{sign}( \sigma^\dagger_i -  \hat \sigma_i) = (-1)^i$ for every $2\leq i\leq n$, that is to say $\sigma_2^\dagger>\hat \sigma_2^\dagger$, $\sigma_3^\dagger<\hat \sigma_3^\dagger$, and so on.
\end{itemize}

\begin{proposition}
    Assume that conditions (i) and (ii) hold. Then, for every $n\geq 2$ and $m=n$, the least squares objective $f:\sigma\mapsto(1/2)\|\Lambda(\sigma)-\Lambda(\sigma^\dagger)\|_2^2$ has $\sigma^\dagger$ as its  unique critical point in $(\RR_{>0})^n$.
    \label{prop_absence_critical_point_d}
\end{proposition}
\begin{proof}
    Just as in the $n=2$ setting, since condition (i) guarantees that $\Lambda'(\sigma)$ is invertible, it is sufficient to show that $\Lambda$ is injective. We prove the injectivity by induction on $n$. For $n=2$, this follows by Proposition~\ref{prop:2d_inj}. Assume now that the result is true for $n-1$. Suppose by contradiction that there exist distinct $\sigma^\dagger$ and $\hat \sigma$ such that $\Lambda(\sigma^\dagger) = \Lambda(\hat \sigma)$. Note that $\sigma^\dagger_i\neq \hat \sigma_i$ for all $i$, otherwise, this contradicts our induction assumption on the injectivity of $\Lambda$. Without loss of generality, we assume that $\hat \sigma_1>\sigma_1^\dagger$. By the alternating signs condition (ii), we have $\sigma_2^\dagger>\hat \sigma_2^\dagger$, $\sigma_3^\dagger<\hat \sigma_3^\dagger$, and so on.

    Let $I\eqdef [\sigma_1^\dagger,\hat \sigma_1]$. We first claim that for all $\sigma_1\in I$, there exists a differentiable function $g\colon I \to (\RR_{>0})^{n-1}$ such that 
    $$
    \forall j=1,\ldots,n-1, \quad \lambda_j(\sigma_1, g(\sigma_1))  = \lambda_j(\sigma^\dagger).
    $$
    Certainly, any such $g$ satisfies $g(\sigma^\dagger_1) = (\sigma^\dagger_i)_{i=2}^{n-1}$ and $g(\hat \sigma_1) = (\hat \sigma_i)_{i=2}^{n-1}$.  Note also that  by the implicit function theorem, any such  $g$  has differential
    \begin{equation}\label{eq:derivative_g}
    g'(\sigma_1) = -(\Lambda_{[n-1]}')^{-1} \left(\partial_1 \lambda_j\right)_{j=1}^{n-1},
    \end{equation}
    where, as a slight abuse of notation, we write $\Lambda_{[n-1]}' \eqdef\Lambda_{[n-1]}'(\sigma_1,g(\sigma_1)) $. Note that, by assumption (i), $\Lambda_{[n-1]}'$ is invertible.

    We now claim that $g$ is well defined on $I$  and maps into the cuboid
    $$
    B = [\hat \sigma_2,  \sigma_2^\dagger] \times [\sigma_3^\dagger, \hat \sigma_3] \times [\hat \sigma_4,  \sigma_4^\dagger] \times \cdots .
    $$
    Indeed, since $C = \{\sigma\in(\RR_{>0})^n:\;\forall j=1,\ldots, n-1,\; \lambda_j(\sigma) = \lambda_j(\sigma^\dagger)\}$ is a closed set, the only reason why $g$ might not exist on $I$ is that $g(\sigma_1)_i$ diverges to infinity or converges to $0$. This cannot happen because we can  apply the implicit function theorem at $\hat \sigma$ (and $\sigma^\dagger$) to conclude that $g$ exists on a neighborhood around $\hat \sigma$ (respectively $\sigma^\dagger$). So, there exists $\epsilon>0$ such that $g(\sigma_1^\dagger + \epsilon)$ and $g(\hat \sigma_1-\epsilon)$ are well defined and are strictly inside the set $B$ by the alternating signs assumption.
    Note also that since the set $C$ is closed, if it is made of disjoint closed intervals inside $B$, then we can apply the implicit function theorem to the end points of these intervals to extend $g$. So, there exists a continuously differentiable function $g\colon I\to B$.

    Now consider $h\colon I\to (\RR_{>0})^{n-1}$ defined by $h(\sigma_1) = \lambda_n(\sigma_1, g(\sigma_1))$. We aim to show that $h'(\sigma_1)\neq 0$ for all $\sigma_1 \in I$ which would lead to a contradiction by Rolle's theorem since we know that $h(\sigma_1^\dagger) = h(\hat \sigma_1) = \lambda_n(\sigma^\dagger)$. Indeed, 
    $$
    h'(\sigma_1) = \partial_1 \lambda_n + \langle{(\partial_i \lambda_n)_{i=2}^{n}},{g'(\sigma_1)}\rangle = \partial_1 \lambda_n - \langle{(\partial_i \lambda_n)_{i=2}^{n}},{(\Lambda'_{[n-1]})^{-1} (\partial_1\lambda_j)_{j=1}^{n-1}}\rangle =\frac{ \mathrm{det}(\Lambda'(\sigma))}{\mathrm{det}(\Lambda_{[n-1]}'(\sigma))},
    $$
    where the final equality is due to the determinant formula for block matrices. Since $\Lambda'(\sigma)$ is invertible, we must have $h'(\sigma_1)\neq 0$ for all $\sigma_1 \in I$.
\end{proof}

\begin{remark}
     We mention here that if condition (i) was relaxed to $\Lambda'(\sigma)$ is invertible for every $n\geq 2$ and every $\sigma \in B\eqdef [a,b]^n$ for some $0<a<b$ and $\sigma^\dagger$ is in the interior of $B$, then our proof implies that the least squares objective $f$ has only one critical point in the interior of $B$.  As we discuss in the following section, in practice, we observe numerically that the Jacobian $\Lambda'(\sigma)$  has non-zero determinant for every $\sigma\in(\RR_{>0})^n$.
\end{remark}

\subsubsection{Discussion on the assumptions} Now, we go back to the two conditions under which the result of \Cref{prop_absence_critical_point_d} holds. 
\paragraph{Condition (i).} 
We first rigorously prove that this condition is satisfied in the case where $\sigma= \mathbf{1}$. The proof makes use of the fact that, in this case, we only need to deal with harmonic functions, whose explicit expression is simple.
\begin{lemma}
    If $\sigma= \mathbf{1}$ and $m=n$ then 
    \begin{equation}\label{eq:sign_det}
    (-1)^{\frac{n(n+1)}{2}}\det \Lambda'(\sigma) >0 \quad \text{and}\quad  (-1)^{\frac{n(n-1)}{2}}\det \Lambda_{[n-1]}'(\sigma) >0.
\end{equation}
\end{lemma}
\begin{proof}
    Since $\sigma= \mathbf{1}$, we have that $u_j(r,\theta) = j^{-1}r^j \cos(j\theta)$ (see Section~\ref{sub:setting}), so that $\|\nabla u_j(r,\theta)\|_2^2=r^{2(j-1)}$, for $j=1,\dots,n$. Thus, by \eqref{eq:derivative_explicit}, we have 
    \[
    \partial_i\lambda_j(\sigma)=-\frac{1}{\pi}\int_{r_i}^{r_{i-1}} \int_0^{2\pi}\|\nabla u_j(r,\theta)\|_2^2 r\,d\theta dr = - j^{-1} (r_{i-1}^{2j}-r_{i}^{2j}) .
    \]
    Let $A_i\in \RR^n$ be the column vector given by $(A_i)_j = R_i^{j}$, where $R_i = r_i^2$. Note that $A_n=0$. Then, using that the determinant is a multilinear function, we obtain
    \[
   \det \Lambda'(\sigma) = \frac{(-1)^n}{n!} \det [A_0-A_1,A_1-A_2,\dots,A_{n-2}-A_{n-1},A_{n-1}]
   = \frac{(-1)^n}{n!} \det [A_0,A_1,\dots,A_{n-2},A_{n-1}].
    \]
    Therefore
\[
\det \Lambda'(\sigma) = \frac{(-1)^n}{n!} \det
\begin{bmatrix}
    R_0 & R_1& \cdots & R_{n-1} \\
    R_0^2 & R_1^2& \cdots & R_{n-1}^2 \\
    \vdots & \vdots & \dots & \vdots \\
    R_0^n & R_1^n& \cdots & R_{n-1}^n 
\end{bmatrix}
= \frac{(-1)^n}{n!} \left(\prod_{i=1}^{n-1} R_i\right)\det
\begin{bmatrix}
    1 & 1& \cdots & 1 \\
    R_0 & R_1& \cdots & R_{n-1} \\
    \vdots & \vdots & \dots & \vdots \\
    R_0^{n-1} & R_1^{n-1}& \cdots & R_{n-1}^{n-1} 
\end{bmatrix}.
\]
The latter is a Vandermonde matrix, for which the determinant can be explicitly computed. More precisely, we have
\[
\det \Lambda'(\sigma) = \frac{(-1)^n}{n!} \left(\prod_{i=1}^{n-1} R_i\right) \prod_{0\le i<j\le n-1} (R_j-R_i)=(-1)^{n+\frac{n(n-1)}{2}} \frac{1}{n!} \left(\prod_{i=1}^{n-1} R_i\right) \prod_{0\le i<j\le n-1} (R_i-R_j).
\]
This concludes the proof of the result concerning the sign of $\det \Lambda'(\sigma)$, since $R_i>R_j$ for $i<j$. 

Concerning the sign of $\det \Lambda'_{[n-1]}(\sigma)$, we stress that removing the row corresponding to the outermost annulus corresponds to removing the column $A_0-A_1$ in the computation above. As a result, the same argument applies.
\end{proof}

Unfortunately, we have been unable to prove this rigorously  in the general case. This may require studying  the dependence of the integrals of $\|\nabla u_j\|_2^2$ on $j$ and on the radii of the subdomains quantitatively, see \cite{garofalo-lin-1986,dicristo-rondi-2021}. As stated in \Cref{subsec_numerical_study_ip}, we always observe numerically that \eqref{eq:sign_det} holds, regardless of the value of $n$ and of $\sigma$. In \Cref{table_jac_det}, we show the minimum value of $(-1)^{n(n+1)/2}\det \Lambda'(\sigma)$, $(-1)^{n(n-1)/2}\det \Lambda_{[n-1]}'(\sigma)$, $\sigma_{\mathrm{min}}(\Lambda'(\sigma))$ (the smallest singular value of $\Lambda'(\sigma)$) and $\sigma_{\mathrm{min}}(\Lambda_{[n-1]}'(\sigma))$ over $I_{k}^n$, where $I_k=\{a+(i/(k-1))(b-a):0\leq i\leq k-1\}$ with $k=5$ and $(a,b)=(0.5, 1.5)$. The minimum of $(-1)^{n(n+1)/2}\det \Lambda'(\sigma)$ seems to be reached at $b\mathbf{1}$, so that increasing $b$ would result in an even smaller value for the minimum.
\begin{table}
\setlength\extrarowheight{1pt}
\begin{center}
\begin{tabular}{|l||c|c|c|c|c|} 
 \hline
 Number of annuli $n$ & 2 & 3 & 4 & 5 & 6\\
 \hline
 Order of min.\ of  $(-1)^{n(n+1)/2}\det \Lambda'(\sigma)$ & $10^{-2}$ & $10^{-4}$ & $10^{-7}$ & $10^{-11}$ & $10^{-15}$\\
  \hline
 Order of min.\ of  $(-1)^{n(n-1)/2}\det \Lambda_{[n-1]}'(\sigma)$ & $10^{-1}$ & $10^{-2}$ & $10^{-4}$ & $10^{-7}$ & $10^{-11}$\\
 \hline
  Order of min.\ of  $\mathrm{min}_{1\leq k\leq n}\,(-1)^{n(n-1)/2}\det M_k(\sigma)$ & $10^{-1}$ & $10^{-3}$ & $10^{-6}$ & $10^{-9}$ & $10^{-13}$\\
 \hline
Order of min.\ of  $\sigma_{\mathrm{min}}(\Lambda'(\sigma))$ & $10^{-2}$ & $10^{-3}$ & $10^{-4}$ & $10^{-5}$ & $10^{-6}$\\ 
 \hline
 Order of min.\ of  $\sigma_{\mathrm{min}}(\Lambda_{[n-1]}'(\sigma))$ & $10^{-1}$ & $10^{-2}$ & $10^{-3}$ & $10^{-4}$ & $10^{-5}$\\ 
 \hline
\end{tabular}
\end{center}
\caption{order of the minimum value of $(-1)^{n(n+1)/2}\det \Lambda'(\sigma)$, $(-1)^{n(n-1)/2}\det \Lambda_{[n-1]}'(\sigma)$, $\mathrm{min}_{1\leq k\leq n}\,(-1)^{n(n-1)/2}\det M_k(\sigma)$, $\sigma_{\mathrm{min}}(\Lambda'(\sigma))$ and $\sigma_{\mathrm{min}}(\Lambda_{[n-1]}'(\sigma))$ over a discretization of $[a,b]^n$ as a function of $n$, with $m=n$ and $(a,b)=(0.5, 1.5)$.}
\label{table_jac_det}
\end{table}

\paragraph{Condition (ii).}
The alternating signs condition can be shown to hold \emph{locally} under a slightly stronger Jacobian invertibility assumption, which we also verified numerically for small values of $n$ (see \Cref{table_jac_det}): \textit{for every ${n\geq 1}$, every $\sigma\in (\RR_{>0})^n$ and every $k\in \{1,\ldots, n\}$, $\mathrm{det}(M_k(\sigma))$ has constant sign (independent of $k$ and  $\sigma$), where ${M_k(\sigma) \in \RR^{(n-1)\times (n-1)}}$ is obtained by removing the last row and the $k$-th column of $\Lambda'(\sigma)$.
}

As in the proof of Proposition~\ref{prop_absence_critical_point_d}, by the implicit function theorem, there exists a $C^1$ function $g$ around $\sigma^\dagger_1$ taking values in $(\RR_{>0})^{n-1}$ such that
$$
\forall j=1,\ldots, n-1, \quad \lambda_j(\sigma_1,g(\sigma_1)) = \lambda_j(\sigma^\dagger).
$$ 
Moreover, locally around $\sigma^\dagger$, its derivative satisfies \eqref{eq:derivative_g}. By Cramer's rule, $g'(\sigma_1)_k = -\det(A_k)/\det(\Lambda_{[n-1]}')$ where $A_k$ is $\Lambda_{[n-1]}'$ with its $k$-th column replaced by $(\partial_1 \lambda_j)_{j=1}^{n-1}$.
By assumption, $\det(A_1)$ and $\det(\Lambda_{[n-1]}')$ have the same sign; $\det(A_2)$ and $\det(\Lambda_{[n-1]}')$  have opposite signs (since we need to swap the first and second columns of $A_2$ to have the columns and rows correspond to increasing annulus position and increasing measurement frequency). Likewise, $\det(A_3)$ and $\det(\Lambda_{[n-1]}')$  have the same signs (two column swaps needed) and so on. As a result, we obtain that $\mathrm{sign}(g_k'(\sigma_1)) = (-1)^k$.

\section{A convex programming approach}
\label{sec_convex}
In the last decades, various convexification strategies have been proposed \cite{klibanov2017globally,klibanovConvexificationElectricalImpedance2019}. More recently, \cite{harrachCalderonProblemFinitely2023} showed that the Calder\'on problem can be formulated as a convex nonlinear semidefinite program (SDP). To our knowledge, this is the first convex reformulation that is valid in the finite measurements and noisy setting. Given our finding that the least-squares objective has no spurious critical points, it is natural to compare the performance of the two approaches. As pointed out in \Cref{sec_previous}, the method introduced in \cite{harrachCalderonProblemFinitely2023} has not been implemented numerically, as the main result of this work is not constructive and how to find the minimal number of measurements and the parameters defining the convex program is not discussed. In this section, focusing on the case of radial piecewise constant conductivities, we propose a way to estimate these parameters and to solve the convex nonlinear SDP.

\subsection{Description of the approach}
Given an unknown conductivity $\sigma^\dagger$ and a priori bounds $a,b\in\RR_{>0}$ such that $a<b$ and $\sigma^\dagger\in [a,b]^n$, it was proposed in \cite{harrachCalderonProblemFinitely2023} to recover $\sigma^\dagger$ by solving the following problem
\begin{equation}
        \underset{\sigma\in[a,b]^n}{\mathrm{min}}~\langle c,\sigma\rangle~~\mathrm{s.t.}~~\Lambda(\sigma)\leq y,
        \tag{$\mathcal{P}_c(y)$}
        \label{problem_convex}
\end{equation}
where $c\in (\RR_{>0})^n$ is a suitably chosen weight vector and $y=\Lambda(\sigma^\dagger)$. Owing to the convexity of $\Lambda$, problem \Cref{problem_convex} is \emph{convex}. This is particularly appealing, as the nonlinearity of the forward map typically leads to the resolution of non-convex optimization problems (such as nonlinear least squares), which are in principle considerably harder to solve.

The main contribution of \cite{harrachCalderonProblemFinitely2023} is to show that, if $m$ is sufficiently large, there \emph{exists} a vector $c\in(\RR_{>0})^n$ such that, for every unknown conductivity $\sigma^\dagger\in[a,b]^n$, problem \Cref{problem_convex} with $y=\Lambda(\sigma^\dagger)$ has a unique solution, which is $\sigma^\dagger$ (see Theorem 3.2 in the above reference). In the following, we say that such a vector $c$ is \emph{universal}. The main issue with this result is that its proof is not constructive, and hence, $c$ is unknown in practice.

\subsection{Numerical implementation}
We now present a numerical procedure to estimate a universal vector $c$, which can then be used to investigate the properties of \eqref{problem_convex} for practical purposes. We first discuss an equivalent characterization of universality of the weight vector $c$ by first order optimality conditions and will exploit this characterization.

\paragraph{Optimality conditions.} Since problem \Cref{problem_convex} is convex, it is natural to investigate whether strong duality holds (we refer the reader to \citet[Section 5]{boydConvexOptimization2004} for more details on strong duality, constraint qualification and optimality conditions). If $y=\Lambda(b\mathbf{1})$, then \Cref{problem_convex} has a unique admissible point, which is $b\mathbf{1}$. Otherwise, one can show that Slater's condition holds, which implies that strong duality holds. As a result, we obtain the following necessary and sufficient first-order optimality conditions.
\begin{lemma}
    Let $\sigma\in [a,b]^n$ and $y\in\RR^m$ be such that $y_j>\lambda_j(b\mathbf{1})$ for every $1\leq j \leq m$. Then $\sigma$ is a solution to \Cref{problem_convex} if and only if there exists $z\in\RR^m$ and $\lambda,\mu\in \RR^n$ such that
    \begin{equation}
        \left\{
        \begin{aligned}
            & \lambda,\mu\geq 0~\mathrm{and}~z\geq 0,\\
            & \lambda_i (\sigma_i-a)=0~\mathrm{and}~\mu_i(\sigma_i-b)=0~\mathrm{for~every}~1\leq i\leq n,\\
            & z_j(\lambda_j(\sigma)-y_j)=0~\mathrm{for~every}~1\leq j\leq m,\\
            &\textstyle c+\mu-\lambda+\sum_{j=1}^m z_j \nabla \lambda_j(\sigma)=0.
        \end{aligned}
        \right.
        \label{opt_cond}
    \end{equation}
    \label{lemma_opt_cond}
\end{lemma}
\begin{proof}
    Since $\Lambda$ is continuous, there exists $\delta>0$ such that, for every $\sigma\in(\RR_{>0})^n$ with $\|\sigma-b\mathbf{1}\|_{\infty}\leq \delta$, it holds 
    $$\lambda_j(\sigma)-\lambda_j(b\mathbf{1})\leq [y_j-\lambda_j(b\mathbf{1})]/2,\qquad 1\leq j\leq m.$$
    Taking $\sigma=(b-\delta)\mathbf{1}$, we obtain for every $1\leq j\leq m$ that 
    \begin{equation*}
        \begin{aligned}
            y_j-\lambda_j(\sigma)&=y_j-\lambda_j(b\mathbf{1})+\lambda_j(b\mathbf{1})-\lambda_j(\sigma)\\
            &\geq (y_j-\lambda_j(b\mathbf{1}))/2\\
            &>0.
        \end{aligned}
    \end{equation*}
    As a result, we have proved the existence of $\sigma\in (a,b)^n$ such that $\lambda_j(\sigma)<y_j$ for every $1\leq j\leq m$, which shows that Slater's condition holds. The necessary and sufficient optimality condition then follows from the convexity of $\Lambda$.
\end{proof}

\subsubsection{Estimation of the weight vector}
\label{subsec_c_estim}
We propose to estimate $c$ by exploiting the optimality conditions \eqref{opt_cond}. To be more precise, we fix $n_{\sigma}$ a set of conductivities $\{\sigma_l\}_{1\leq l\leq n_{\sigma}} \subset [a,b]^n$ and try to find a vector $c$ such that $\sigma_l$ is a solution to \Cref{problem_convex} with $y=\Lambda(\sigma_l)$ for every $1\leq l\leq n_{\sigma}$. By \Cref{lemma_opt_cond}, this is equivalent to the existence, for every $1\leq l\leq n_{\sigma}$, of $(\lambda_l,\mu_l,z_l)\in \RR^n\times\RR^n\times\RR^m$ such that \eqref{opt_cond} holds. The crucial point is that this constraint is \emph{linear} in $c$. As a result, given the problem parameters $(n,m,a,b)$ and $\{\sigma_l\}_{1\leq l\leq n_\sigma}\subset [a,b]^n$, one can check that $c$ is such that every $\sigma_l$ is a solution to \Cref{problem_convex} with $y=\Lambda(\sigma_l)$ by solving a \emph{linear feasibility problem}, for which efficient solvers exist. If this holds and $\{\sigma_l\}_{1\leq l\leq n_\sigma}$ is a sufficiently fine discretization of $[a,b]^n$, one can moreover hope that the estimated $c$ is universal. The resulting linear feasibility problem reads as follows:
\begin{equation}
    \begin{aligned}
        \underset{\substack{c\in \RR^n\\\lambda,\mu\in\RR^{n_{\sigma}\times n}\\z\in\RR^{n_{\sigma}\times m}}}{\mathrm{max}}\quad&c_n&&\\
        \mathrm{s.t.}\quad&c\geq 0~\mathrm{and}~c_1=1,&&\\
                      &\lambda,\mu\geq 0~\mathrm{and}~z\geq 0,&&\\
                      &\lambda_{l,i}(\sigma_{l,i}-a)=0~\mathrm{and}~\mu_{l,i}(\sigma_{l,i}-b)=0,&&1\leq i\leq n,1\leq j\leq n_{\sigma},\\
                      &c_i+\mu_{l,i}-\lambda_{l,i}+\sum\limits_{j=1}^m z_{l,j}\partial_i\lambda_j(\sigma_l)=0,&&1\leq i\leq n,~1\leq l\leq n_{\sigma}.
    \end{aligned}
    \label{c_estimation_problem}
\end{equation}
The constraint $c_1=1$ allows us to circumvent the fact that solutions of \eqref{problem_convex} are unchanged by the multiplication of $c$ by a positive constant. The maximization of $c_n$ is optional (that is to say, one can simply find a feasible tuple $(c,\lambda,\mu,z)$). Maximizing $c_n$ allows to ensure that, if $\hat{c}$ is a solution to \eqref{c_estimation_problem}, then any universal vector $c$ must satisfy $c_n\leq \hat{c}_n$, as the constraints of \eqref{c_estimation_problem} is a subset of the constraints that any universal weight vector must satisfy. As a result, maximizing $c_n$ gives us an upper bound on the last coordinate of any universal weight vector.

\paragraph{Numerical evaluation.} We take $(a,b)=(0.5, 1.5)$ and assess the performance of this estimation procedure as follows. We take $\{\sigma_l\}_{1\leq l\leq n_\sigma}$ to be $(I_k)^n\setminus\{b\mathbf{1}\}$ where $I_k=\{a+(i/(k-1))(b-a):0\leq i\leq k-1\}$ and $k=5$. We solve the feasibility problem for increasing values of $m$ starting from $m=n$ until it becomes feasible. We use the package \texttt{JuMP.jl} \citep{Lubin2023} to model the feasibility problem, which allows us to try a large number of solvers seamlessly. The results of this experiment are reported in \Cref{table_c_estimation_1}. The coefficients of the estimated vectors $c$ are displayed in \Cref{fig_c_estimation_1}. The minimal number of measurements seems to grow linearly with $n$ and is significantly greater than $n$ even for small values of $n$. The smallest value of $c$ (which is always $c_n$, the value corresponding to the innermost annulus) rapidly decreases with $n$ and becomes very small even for small values of $n$. This means that the estimation of an unknown conductivity with this SDP approach will necessarily enjoy poor stability, as large variations of $\sigma_n$ will result in small variations of the objective $\langle c,\sigma\rangle$ of \Cref{problem_convex}. This phenomenon illustrates the ill-posedness of the inverse problem, which increases with the distance to the boundary and is not resolved by this convex programming approach. For $n=6$, all the linear programming solvers we have tried (HiGHS \citep{huangfuParallelizingDualRevised2018}, SCS \citep{ocpb:16}, Clarabel \citep{Clarabel_2024}, Gurobi, MOSEK) struggled to solve the problem. This could be due to the fact that $c_n$ must be close to machine precision for $n=6$. Let us stress that our approach for estimating $c$ is also limited by its computational cost. As an example, the linear program we would need to solve for $(n,m,k)=(7, 27, 5)$ involves more than three million variables with  more than one million constraints. The results of this experiment heavily depend on the choice of the a piori bounds $a$ and $b$ on the unknown conductivity. If instead of taking $(a,b)=(0.5, 1.5)$ as above we take $(a,b)=(0.75,1.25)$, the minimal number of measurements is smaller and the smallest coefficient of $c$ is larger. The results in this setting are reported in \Cref{table_c_estimation_2} and \Cref{fig_c_estimation_2}.

\begin{table}
\setlength\extrarowheight{1pt}
\begin{center}
\begin{tabular}{|l||c|c|c|c|} 
 \hline
 Number of annuli $n$ & 2 & 3 & 4 & 5\\
 \hline
 Min. num. of measurements $m$ & 3 & 7 & 13 & 20\\ 
 \hline
 Smallest coefficient of $c$ & $7.8\cdot10^{-2}$ & $1.5\cdot 10^{-3}$ & $3.6\cdot 10^{-6}$ & $1.1\cdot 10^{-10}$\\
 \hline
\end{tabular}
\end{center}
\caption{minimal number of measurements required for the existence of the vector $c$ and smallest coefficient of $c$ as a function of the number of annuli $n$ ($a=0.5$ and $b=1.5$).}
\label{table_c_estimation_1}
\end{table}

\begin{table}
\setlength\extrarowheight{1pt}
\begin{center}
\begin{tabular}{|l||c|c|c|c|c|} 
 \hline
 Number of annuli $n$ & 2 & 3 & 4 & 5 & 6\\
 \hline
 Min. num. of measurements $m$ & 2 & 5 & 8 & 13 & 19\\ 
 \hline
 Smallest coefficient of $c$ & $1.8\cdot10^{-1}$ & $2.1\cdot 10^{-2}$ & $3.3\cdot 10^{-4}$ & $1.9\cdot 10^{-6}$ & $2.8\cdot 10^{-9}$\\
 \hline
\end{tabular}
\end{center}
\caption{minimal number of measurements required for the existence of the vector $c$ and smallest coefficient of $c$ as a function of the number of annuli $n$ ($a=0.75$ and $b=1.25$).}
\label{table_c_estimation_2}
\end{table}

\begin{figure}
    \centering
    \subfloat[$(a,b)=(0.5,1.5)$]{\label{fig_c_estimation_1}\includegraphics[width=3.2in]{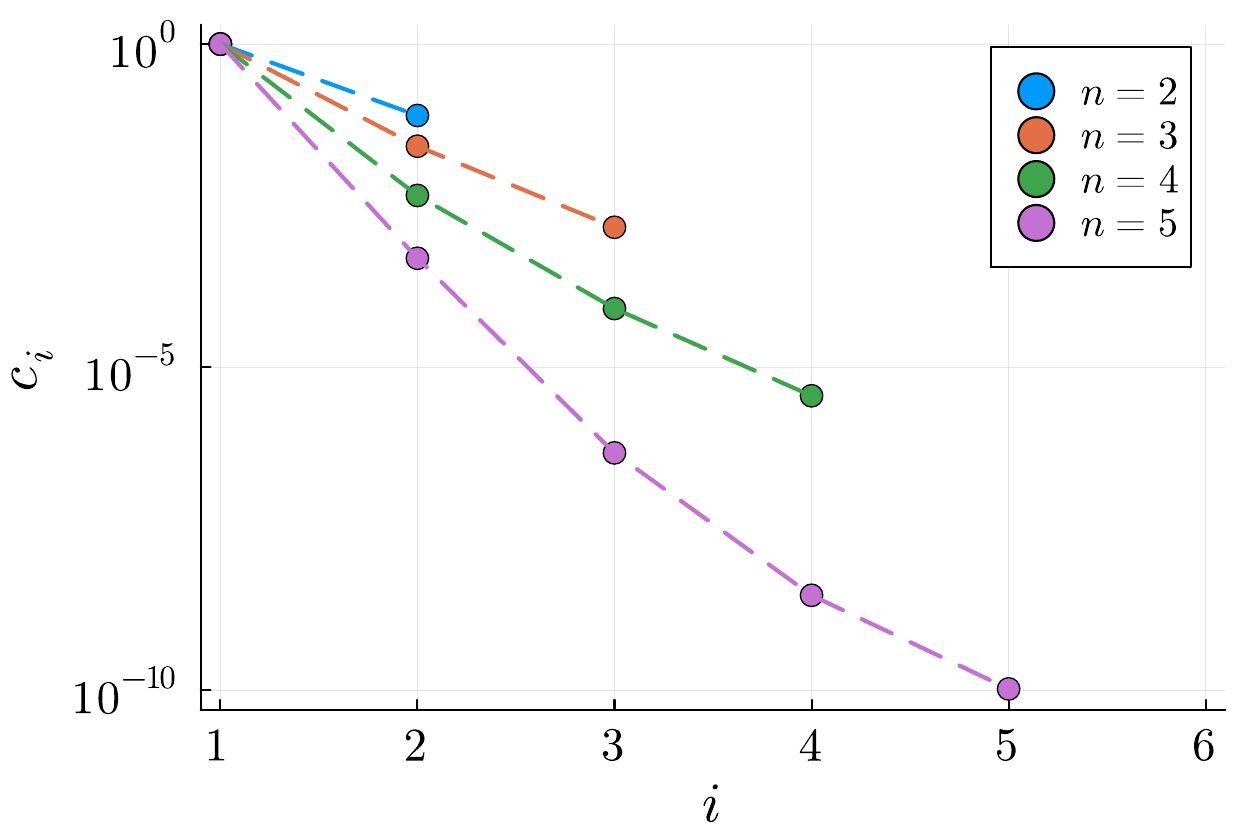}}~
    \subfloat[$(a,b)=(0.75,1.25)$]{\label{fig_c_estimation_2}\includegraphics[width=3.2in]{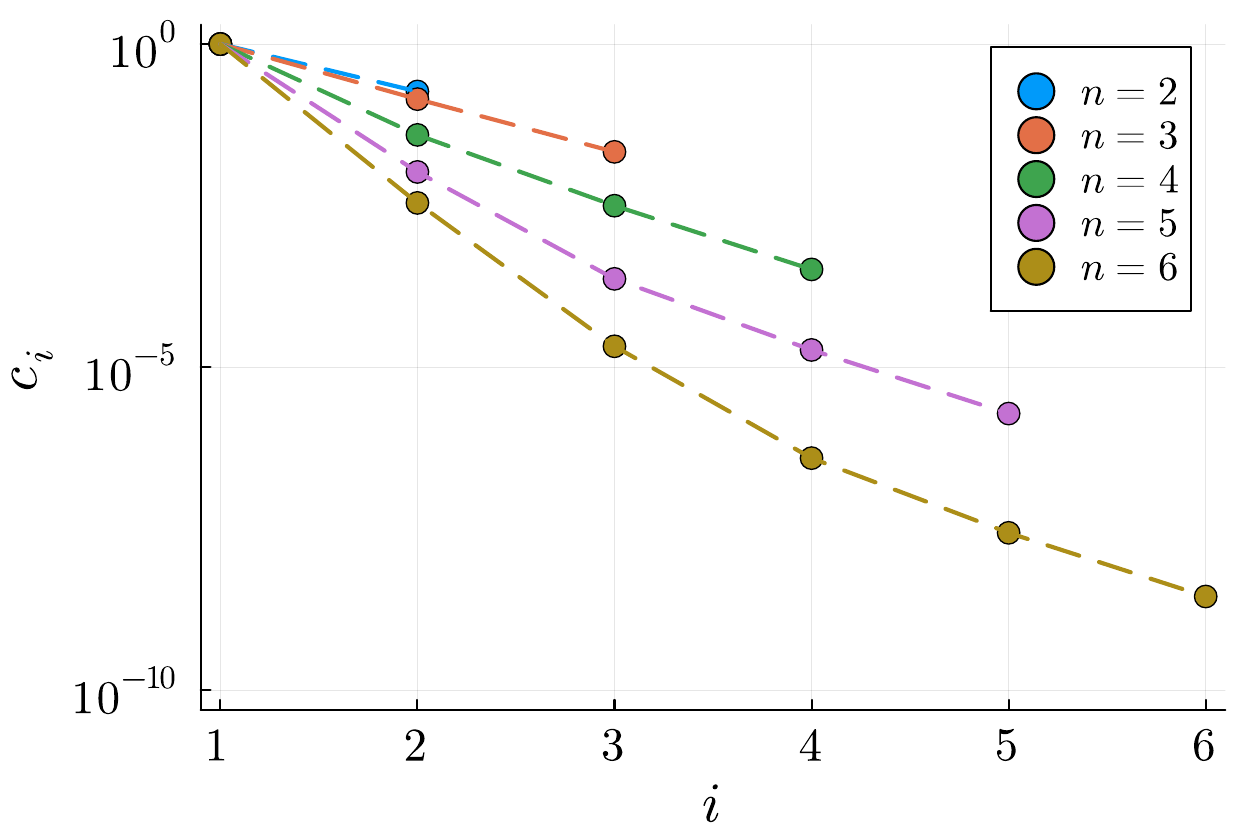}}
    \caption{coefficients of the estimated vector $c$ in two settings. For each $n$, the number of measurements $m$ is the one reported in \Cref{table_c_estimation_1} (a) or \Cref{table_c_estimation_2} (b).}
    \label{fig_c_estimation}
\end{figure}

\paragraph{Handcrafting the weight vector.} As a final experiment, we investigate the universality of several handcrafted weight vectors $c$ with a decay similar to that of the estimated one. We take $n=5$, $m=13$, $(a,b)=(0.75, 1.25)$ and draw $100$ random conductivities uniformly in $[a,b]^n$. We solve \eqref{problem_convex} with the estimated vector $c$ and the vector $c=(10^{k_i})_{1\leq i\leq n}$ with $k_i$ linearly interpolating between $k_1=0$ and $k_n$ for several value of $k_n$ (see \Cref{fig_c_guess} for a plot of the different choices of $c$). The histograms of the $\ell^{\infty}$ estimation error for the different choices of $c$ are displayed in \Cref{fig_err_c_guess}. The first two guesses, which have a faster decay than the estimated $c$, are approximately universal but yield errors that are slightly larger. On the contrary, the last two guesses, which decay slower than the estimated $c$, are not universal (they yield errors that are close to that of random guessing on several instances). Still, on the majority of instances, they allow us to obtain an estimation error which is smaller or approximately equal to the one obtained with the estimated $c$. This suggests that designing the weight vector $c$ by hand instead of relying on the estimation procedure proposed above could be difficult.

\begin{figure}
    \centering
    \includegraphics[width=0.5\linewidth]{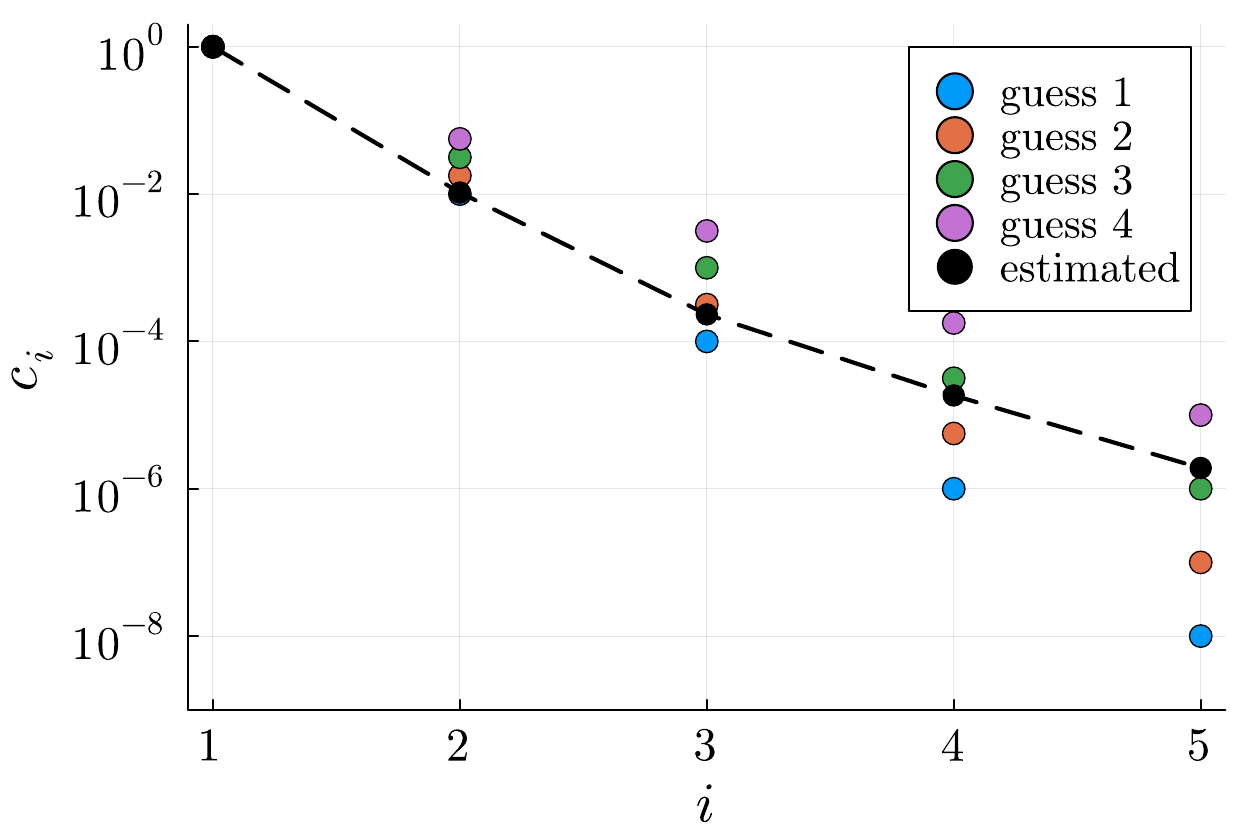}
    \caption{plot of four handcrafted vectors $c$ along with the estimated one for $n=5$, $m=13$ and ${(a,b)=(0.75, 1.25)}$.}
    \label{fig_c_guess}
\end{figure}

\begin{figure}
\centering
\subfloat[estimated]{\includegraphics[width = 2.0in]{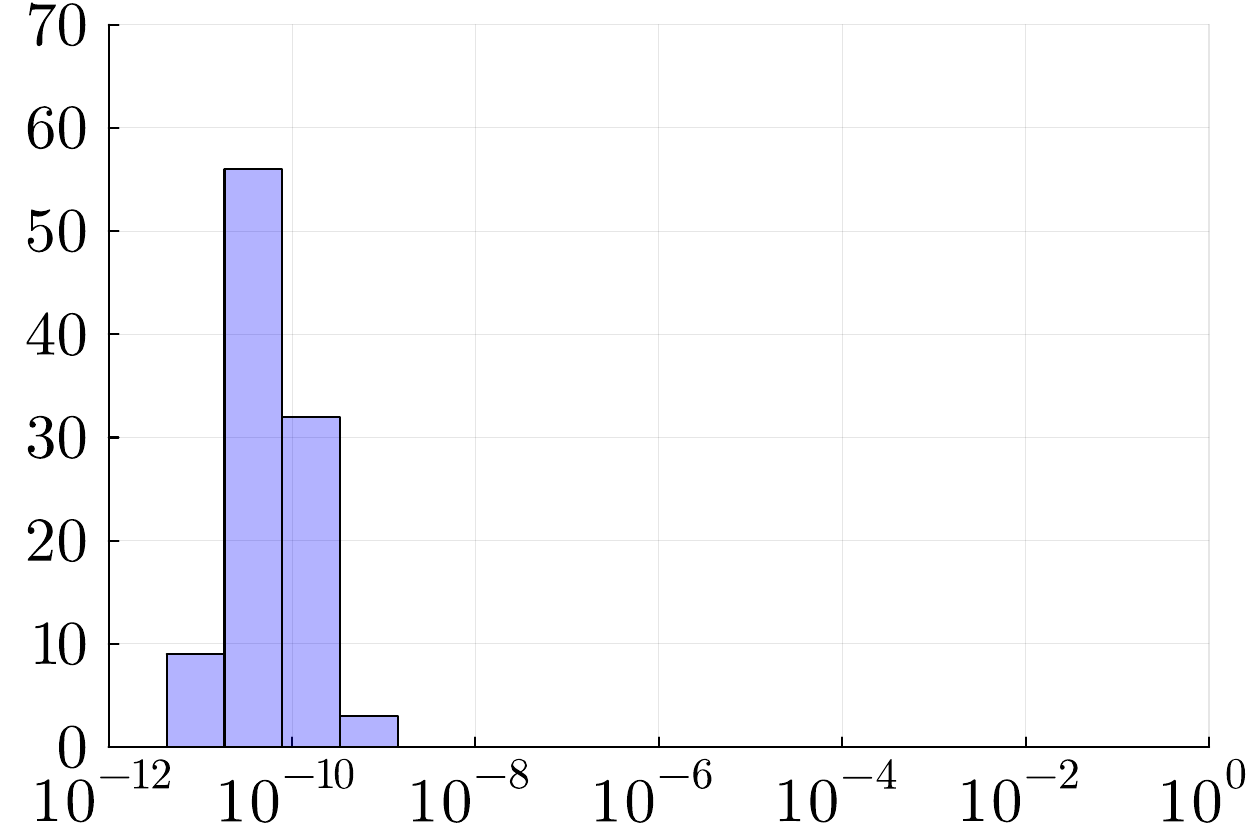}}\quad
\subfloat[guess 1]{\includegraphics[width = 2.0in]{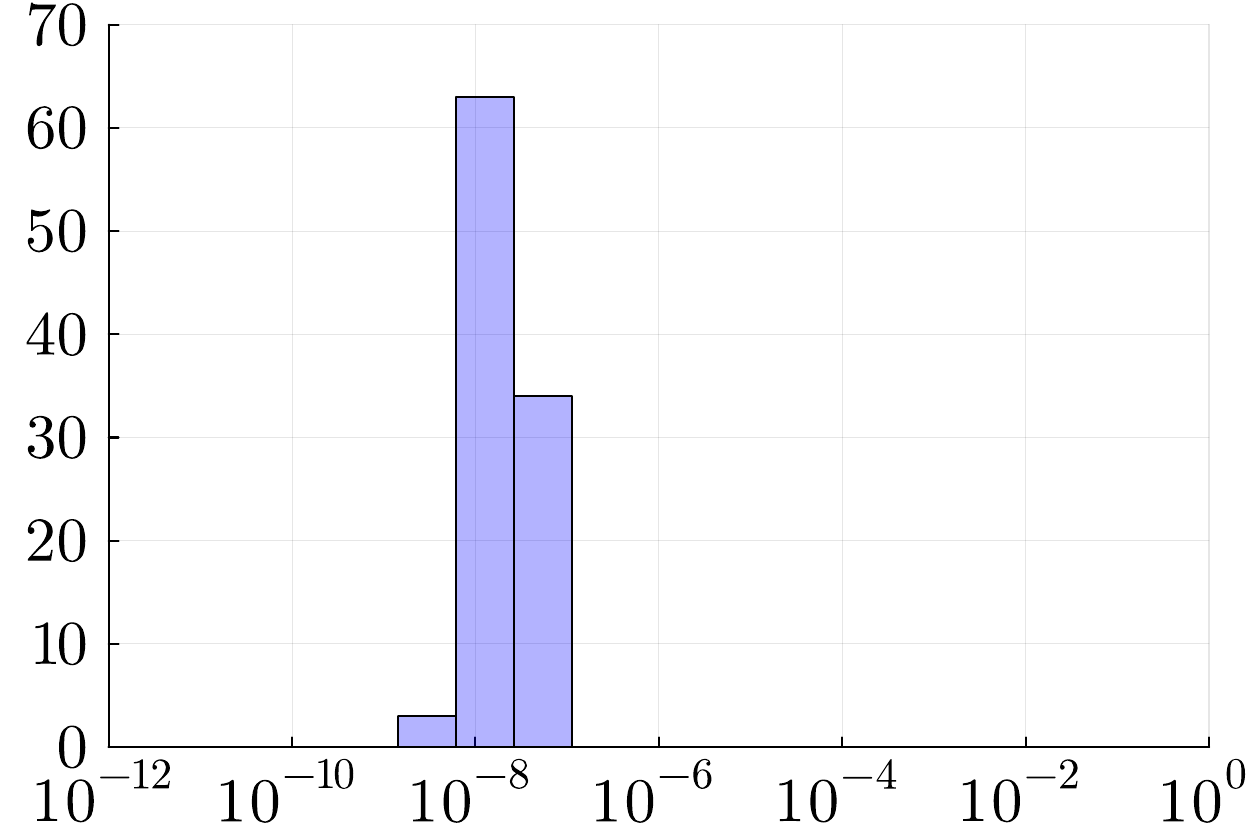}}\quad
\subfloat[guess 2]{\includegraphics[width = 2.0in]{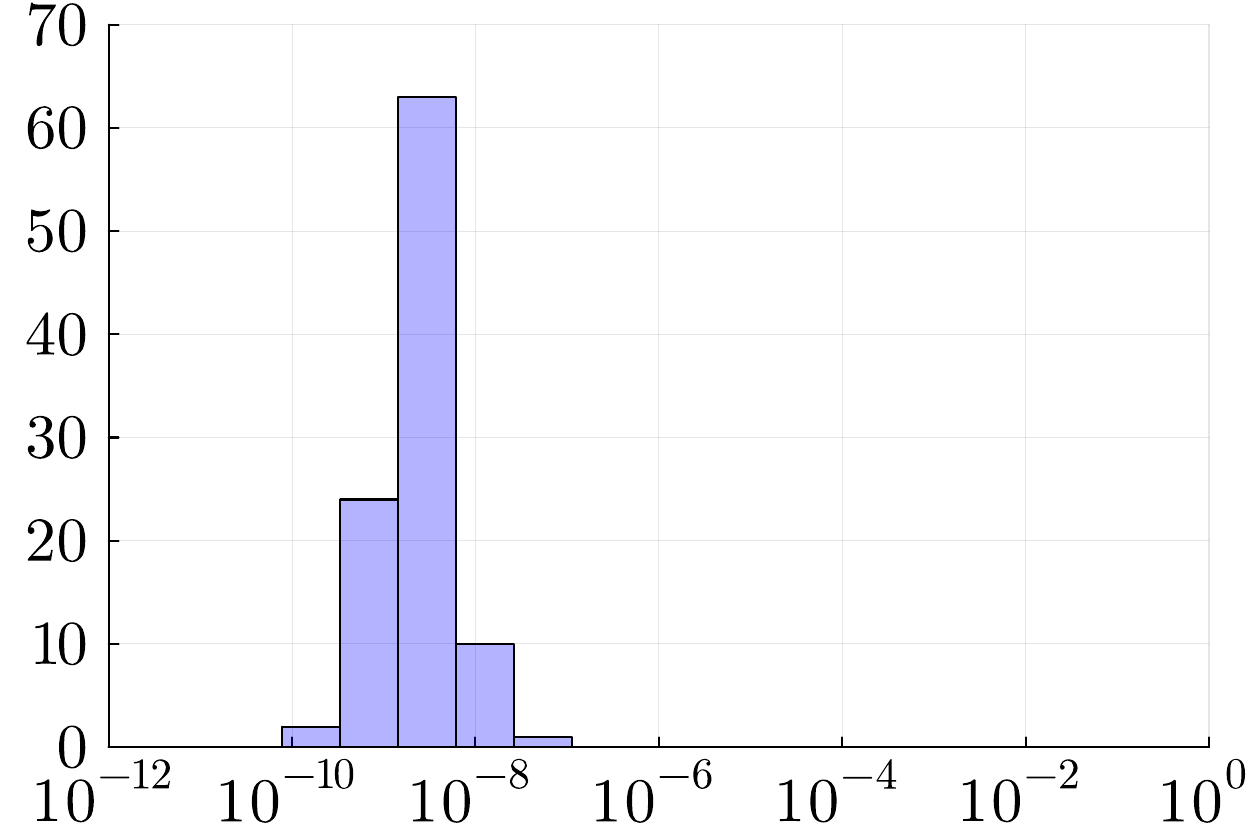}}\\
\hspace{5.4cm}\subfloat[guess 3]{\includegraphics[width = 2.0in]{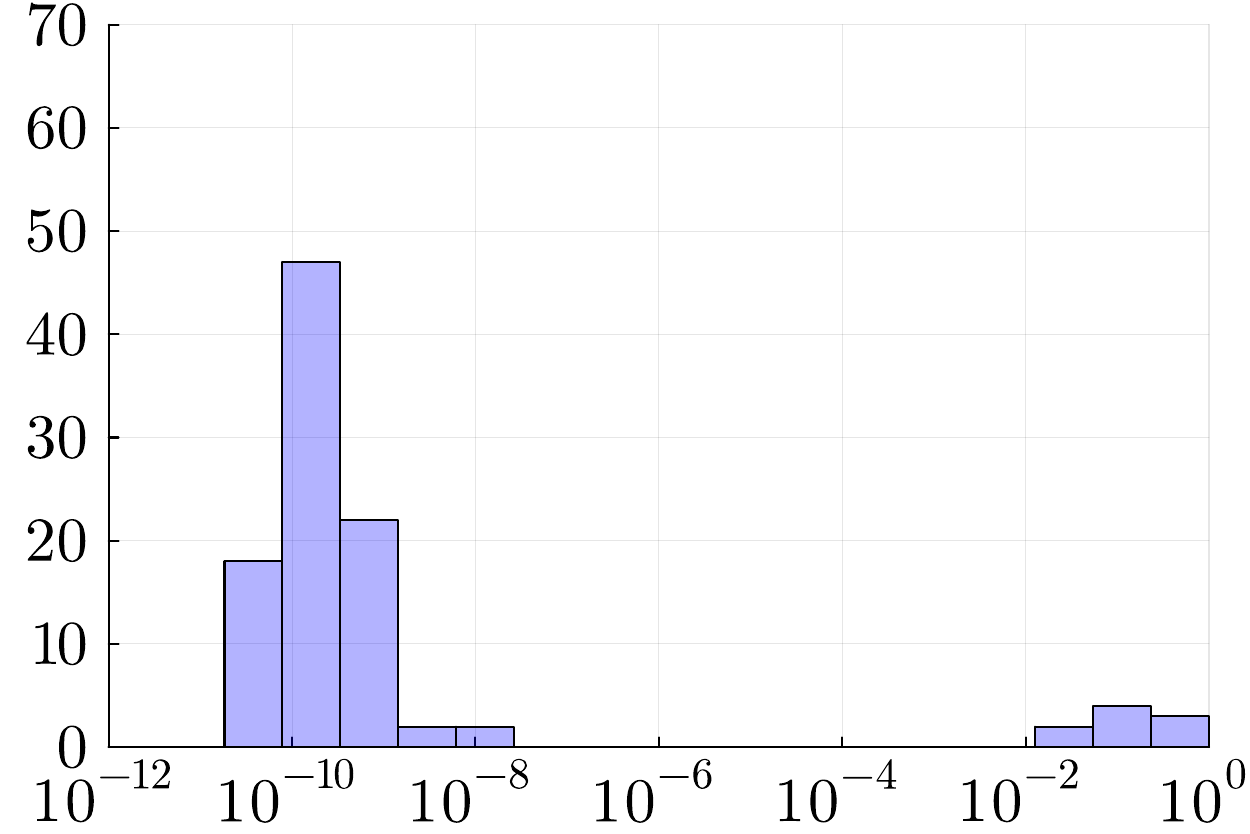}}\quad
\subfloat[guess 4]{\includegraphics[width = 2.0in]{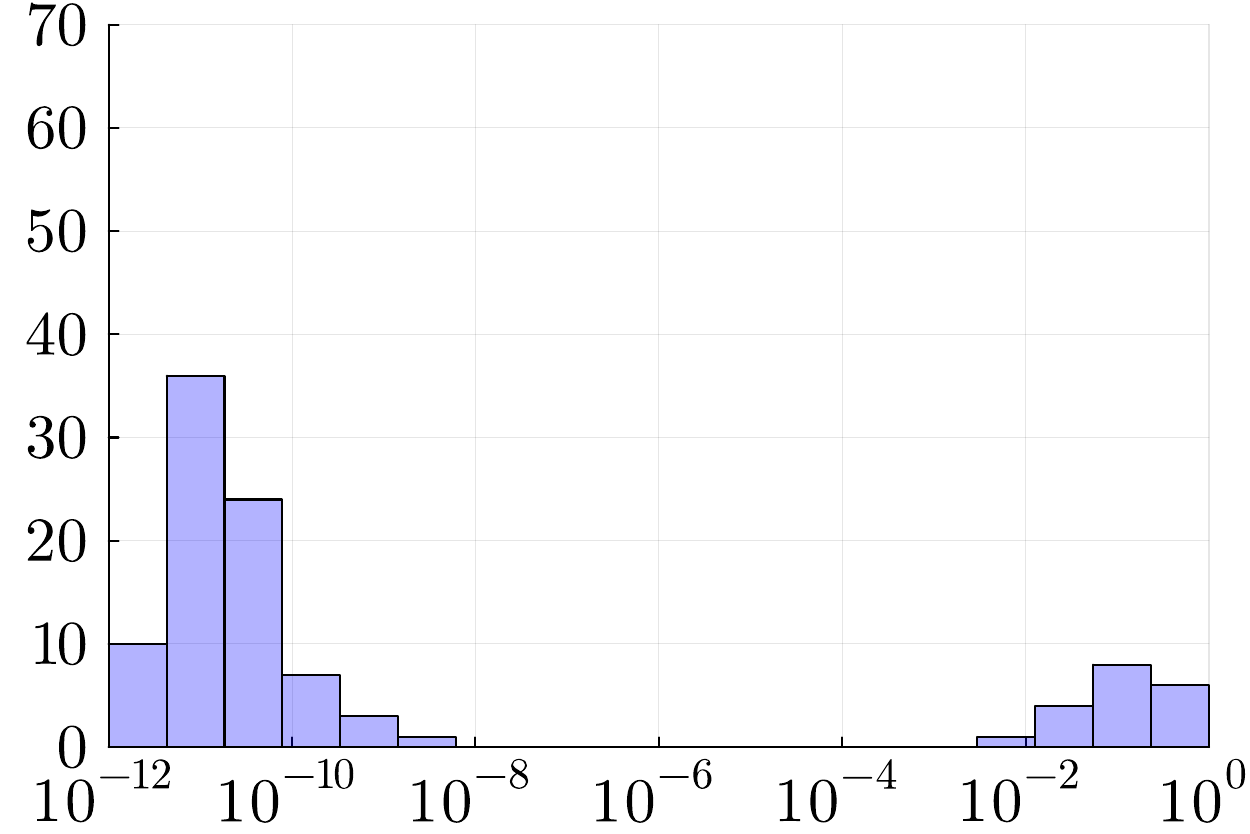}}
\caption{histograms of the $\ell^{\infty}$ estimation error for $n=5$, $m=13$ and $(a,b)=(0.75, 1.25)$ with different choices of $c$ (see \Cref{fig_c_guess}).}
\label{fig_err_c_guess}
\end{figure}

\subsubsection{Numerical resolution of the convex nonlinear semidefinite program}
Using the proposed procedure to estimate a universal weight vector $c$, the resolution of the convex nonlinear semidefinite problem \eqref{problem_convex} can be carried out using standard optimization solvers. Since the first and second order derivatives of the forward map $\Lambda$ can be efficiently computed via automatic differentiation, we propose to solve \eqref{problem_convex} by relying on interior point methods. These algorithms consist in solving \eqref{problem_convex} by applying Newton's method to a sequence of equality constrained problems. We refer the reader to \citet[Chapter 11]{boydConvexOptimization2004} for more details on this family of algorithms. In practice, we use the package \texttt{Optimization.jl}\footnote{\href{https://github.com/SciML/Optimization.jl}{https://github.com/SciML/Optimization.jl}} to model \eqref{problem_convex} and solve it with the Ipopt solver \citep{wachterImplementationInteriorpointFilter2006}.

\section{Numerical comparison of several reconstruction methods}
\label{sec_num_comparison}
In this section, we perform a comparison of various reconstruction methods, both in the case of noiseless and noisy measurements.
\subsection{Comparison of the convex approach with least squares}
In this subsection, we compare the convex nonlinear SDP approach of \cite{harrachCalderonProblemFinitely2023} to the classical least squares approach. In agreement with the results of \Cref{sec_theory}, we find that the latter does not suffer from the problem of local convergence, and that, in addition to requiring less measurements and being faster, it always outperforms the convex nonlinear SDP approach. In what follows, we use Powell's dog leg method to solve the least squares problem. This trust-region method combines the steepest and the Gauss-Newton descent directions, and performed best among all the solvers we tested.

\paragraph{Noiseless setting.} To assess this, we take $(a,b)=(0.5, 1.5)$ and $n\in\{2,3,4,5\}$. For each value of $n$ we take $m$ to be the value reported in \Cref{table_c_estimation_1}. Then, we draw $1000$ unknown conductivities $\sigma^\dagger$ uniformly at random in $[a,b]^n$. On each instance, we compare two methods for estimating $\sigma^\dagger$:
\begin{enumerate}
    \item using the least squares solver initialized uniformly at random in $[a,b]^n$ to find an approximate solution of ${\Lambda(\sigma)=\Lambda(\sigma^\dagger)}$;
    \item solving \Cref{problem_convex} with an interior point method with $y=\Lambda(\sigma^\dagger)$ and $c$ estimated as in \Cref{subsec_c_estim}.
\end{enumerate}
The results of this experiment are displayed in \Cref{fig_sdp_rootfinding_noiseless}. We notice that the least squares approach consistently outperforms the SDP approach. The performance of the latter degrades much faster than the performance of the former as $n$ grows. We also stress that the least squares approach is still able to accurately recover the unknown if $m$ is taken to be equal to $n$ instead of $3$ ($n=2$), $7$ ($n=3$), $13$ ($n=4$) or $20$ ($n=5$). Finally, the least squares approach is significantly faster than the SDP approach, especially for larger values of $n$.

\paragraph{Noisy setting.} We also conduct a similar experiment in the noisy setting, where we only have access to noisy measurements $z=\Lambda(\sigma^\dagger)+w$, where $w\in\RR^m$ is an additive noise. In this case, \cite{harrachCalderonProblemFinitely2023} proposes to solve \Cref{problem_convex} with $y=z+\delta\mathbf{1}$, where $\delta>0$ is an estimate of the noise level. We compare this method to the approach that consists in using the least squares solver initialized uniformly at random in $[a,b]^n$ to minimize $\sigma\mapsto\|\Lambda(\sigma)-z\|_2^2$. The experimental setting is the same as above, the coordinates of the noise being independent and uniformly distributed in $(-\delta,\delta)$ with $\delta=10^{-4}$. The results are reported in \Cref{fig_sdp_leastsquares_noisy}. The difference of the estimation errors is smaller than in the noiseless setting, but the least squares approach still consistently outperforms the SDP approach.

\begin{figure}
\centering
\subfloat[$n=2$, $m=3$]{\includegraphics[width = 2.2in]{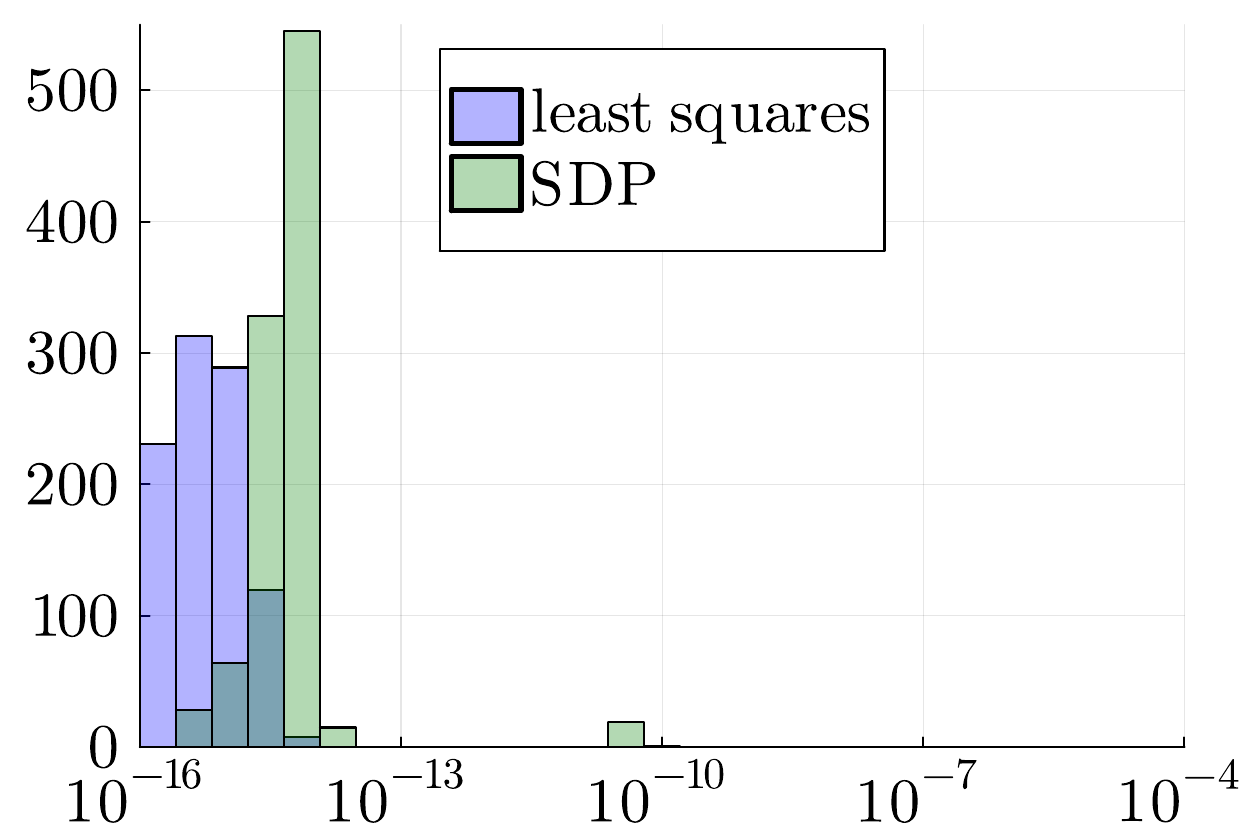}}\qquad\qquad\qquad
\subfloat[$n=3$, $m=7$]{\includegraphics[width = 2.2in]{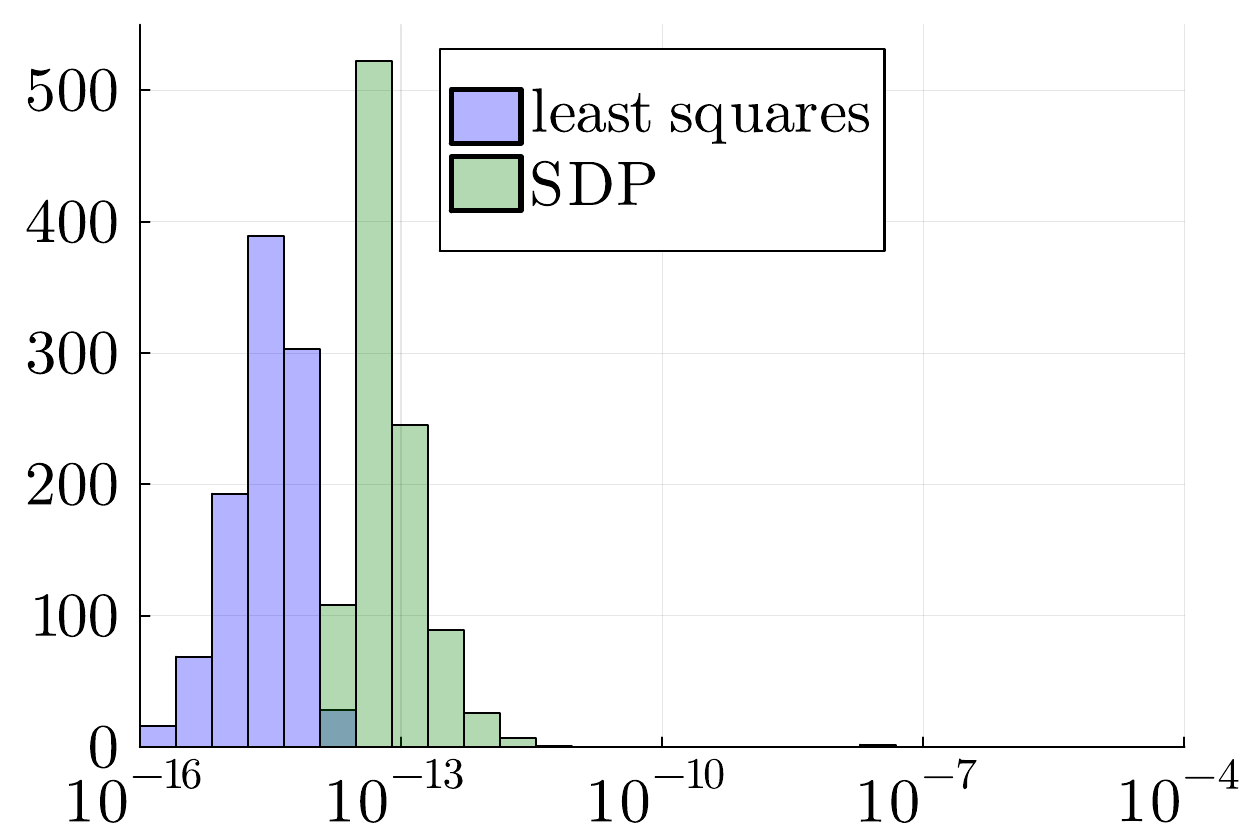}}\\
\subfloat[$n=4$, $m=13$]{\includegraphics[width = 2.2in]{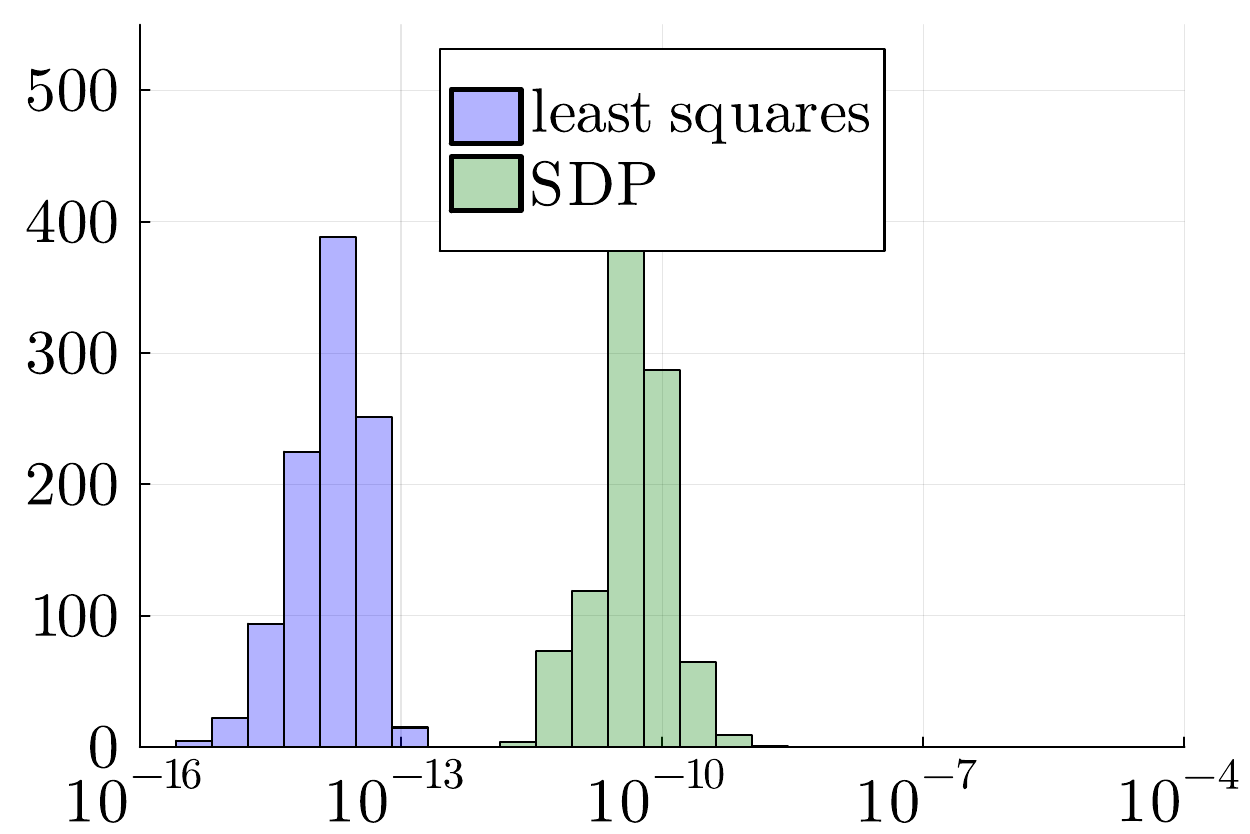}}\qquad\qquad\qquad
\subfloat[$n=5$, $m=20$]{\includegraphics[width = 2.2in]{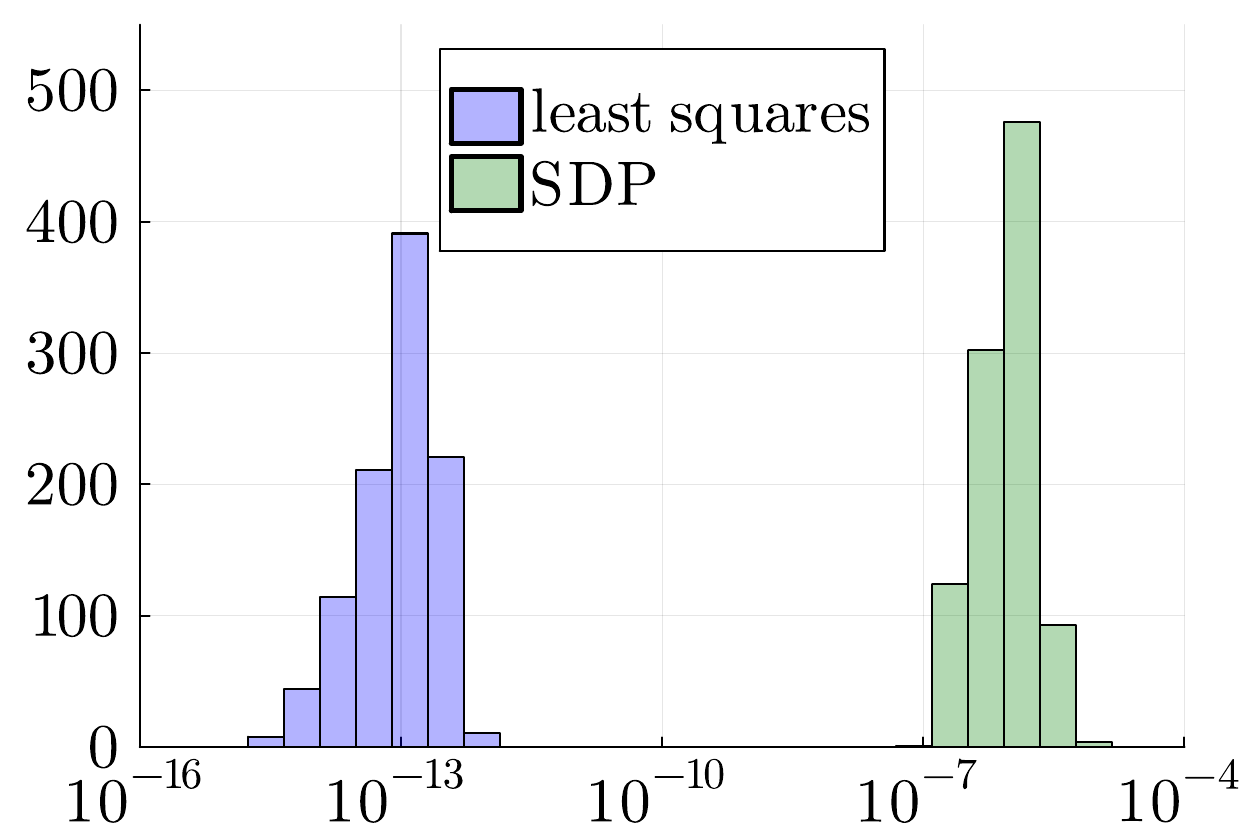}}
\caption{histograms of the $\ell^{\infty}$ estimation error of the least squares and the SDP approaches on 1000 instances with $(a,b)=(0.5, 1.5)$.}
\label{fig_sdp_rootfinding_noiseless}
\end{figure}

\begin{figure}
\centering
\subfloat[$n=2$, $m=3$]{\includegraphics[width = 2.1in]{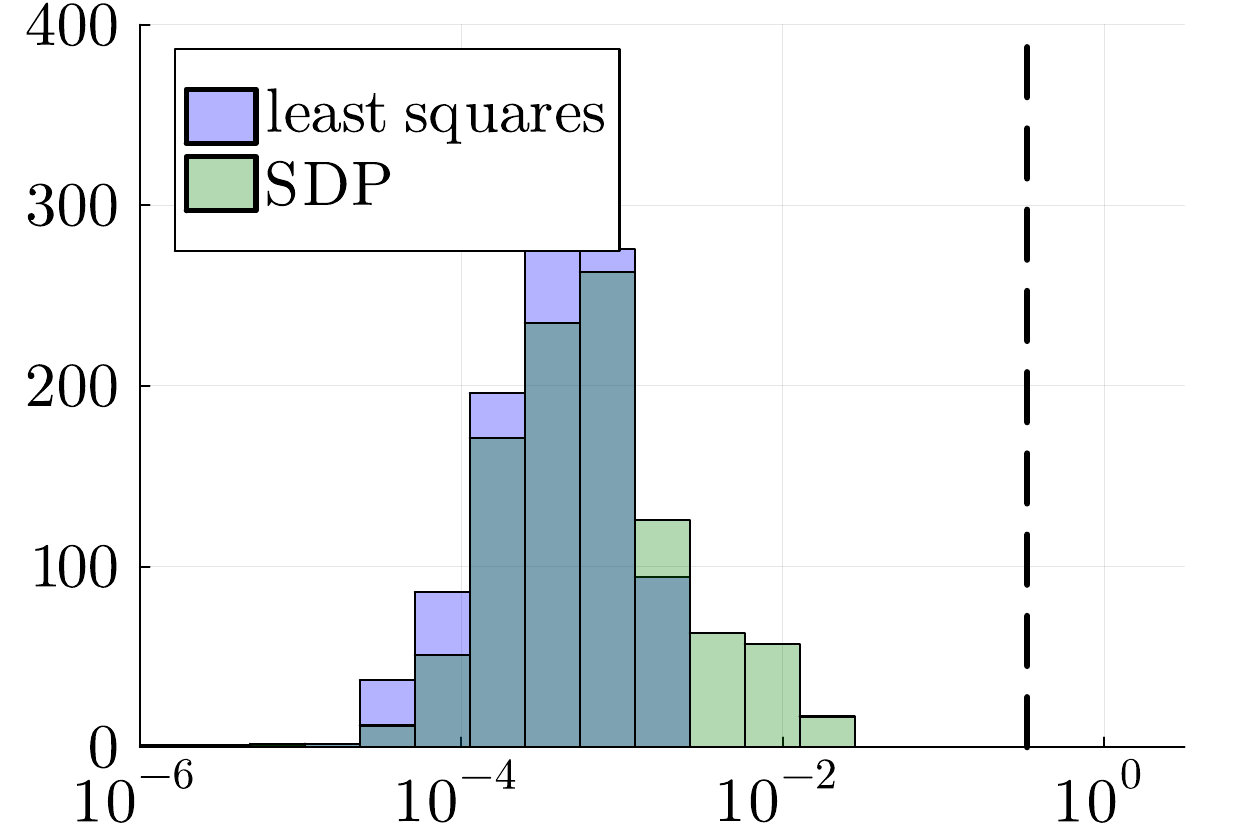}}
\subfloat[$n=3$, $m=7$]{\includegraphics[width = 2.1in]{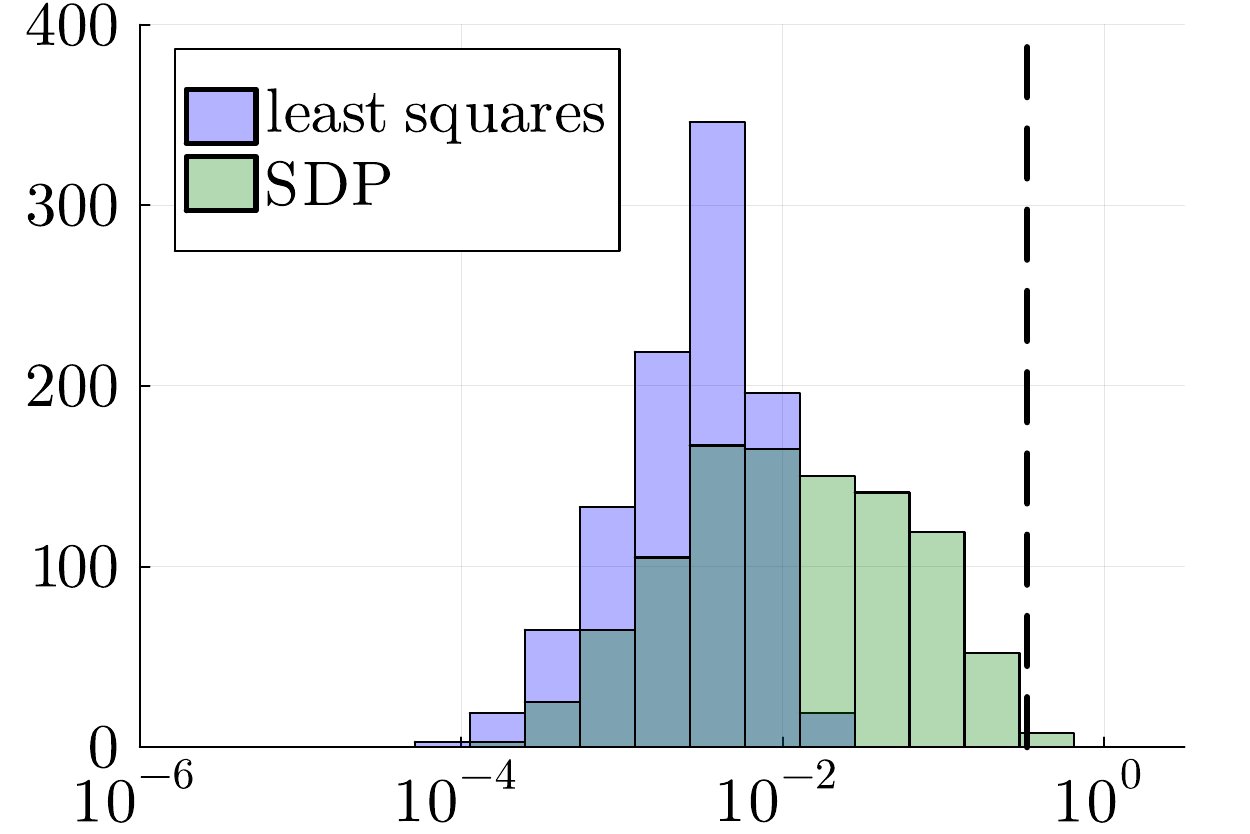}}
\subfloat[$n=4$, $m=13$]{\includegraphics[width = 2.1in]{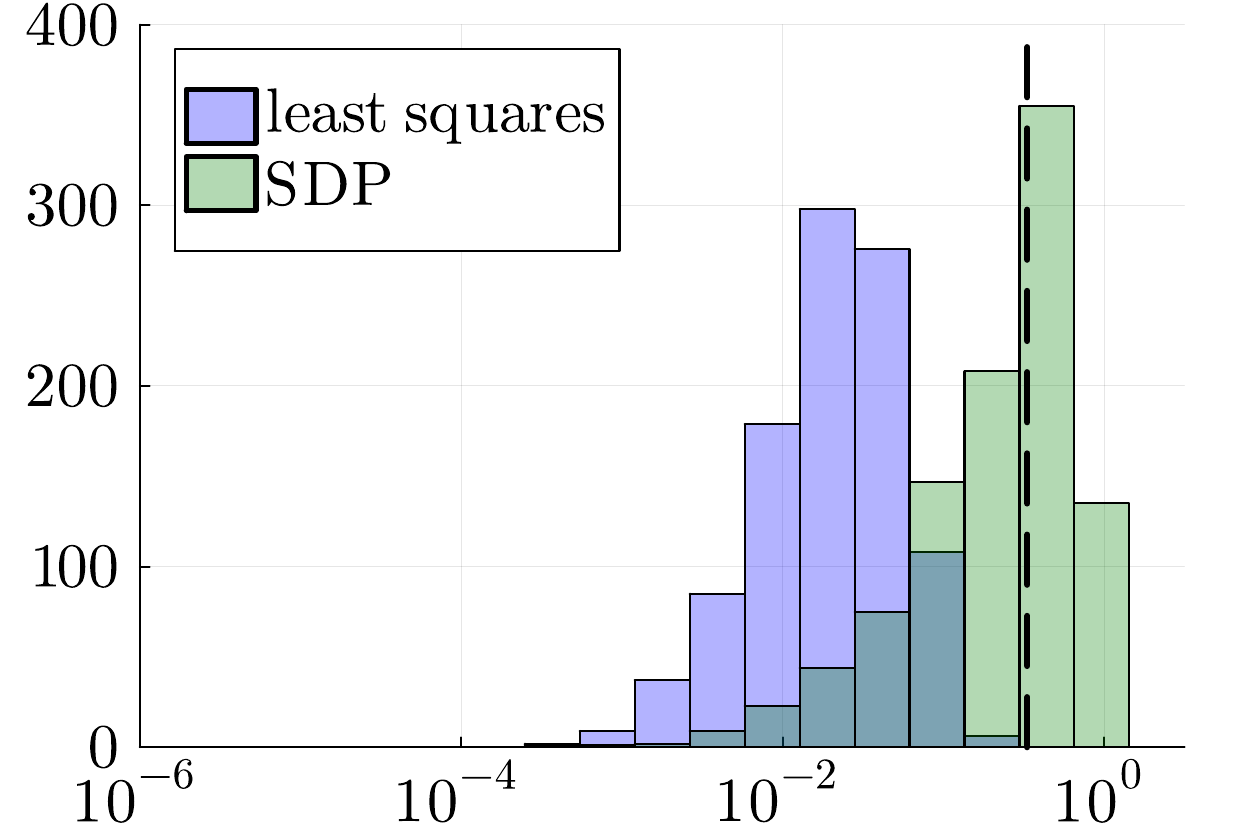}}
\caption{histograms of the $\ell^{\infty}$ estimation error of least squares and the SDP approach in the noisy case on 1000 instances with $(a,b)=(0.5, 1.5)$. The dashed black line indicates the mean estimation error of a naive random guessing strategy, which is $(b-a)/3$.}
\label{fig_sdp_leastsquares_noisy}
\end{figure}

\paragraph{Influence of the box size.} One could argue that enlarging the box size could result in more failures of the least squares approach, as it could suffer from the problem of local convergence. This does not seem to be the case. When the box is enlarged (for example taking $(a,b)=(0.1, 1.9)$), the performance of the nonlinear SDP approach significantly deterioriates: the minimal number of measurements $m$ significantly increases and the smallest value of $c$ is several orders of magnitude smaller. On the other hand, the performance of the least squares approach is almost unchanged.

\paragraph{Convergence of the least squares solver.} When $n\leq 9$, we observed that the least squares solver (Powell's dog leg method) always converges to a global minimum of the least squares objective. When $n\geq 8$, for some worst-case instances, the solver runs for a very large number of iterations, its iterates going far away from the global minimum while keeping a low objective value (the maximum number of iterations over all instances for $n\leq 7$ is $22$ and the histograms of the number of iterations for $n\in\{8,9,10\}$ are given in \Cref{fig_hist_iter}). After a large number of iterations, the iterates finally converge to the true conductivity. This behavior is showcased in \Cref{fig_conv_ls} for a specific instance with $n=m=9$. When $n=10$, the solver fails to converge on very few instances (of the order of $1$ out of $100$), after running for several million iterations. The gradient of the objective at the last iterate is non-zero, but its coordinate corresponding to the innermost conductivity is close to machine precision. Our experiments suggest that allowing for more iterations would lead to convergence. Moreover, this phenomenon is likely caused by the ill-posedness of the inverse problem, which allows conductivities that differ by several orders of magnitude to yield measurements that are close. We also stress that convergence failures are never caused by the existence of spurious critical points. Finally, let us recall that, for $n \geq 6$, the SDP approach becomes infeasible, since the the smallest coordinate of $c$ rapidly reaches machine precision as $n$ increases.

\begin{figure}
\centering
\subfloat[$n=8$]{\includegraphics[width = 2.1in]{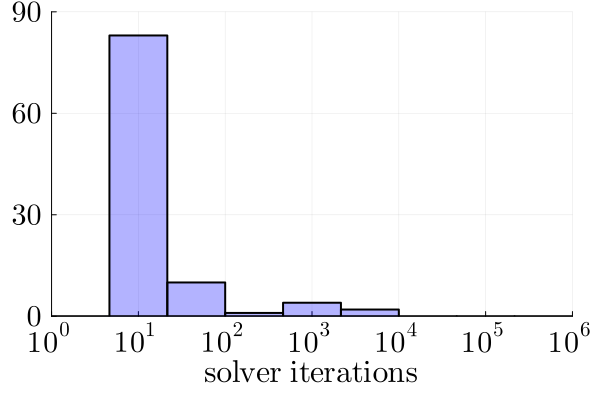}}~
\subfloat[$n=9$]{\includegraphics[width = 2.1in]{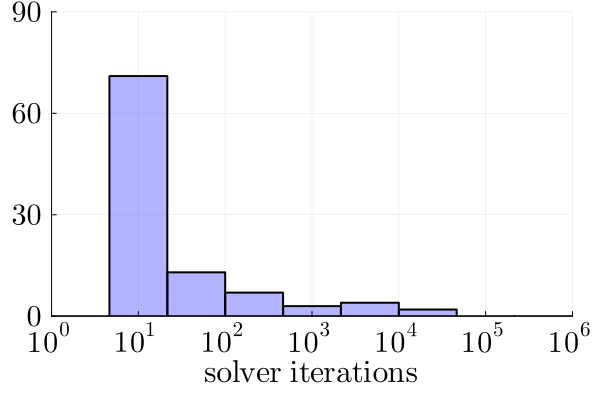}}~
\subfloat[$n=10$]{\includegraphics[width = 2.1in]{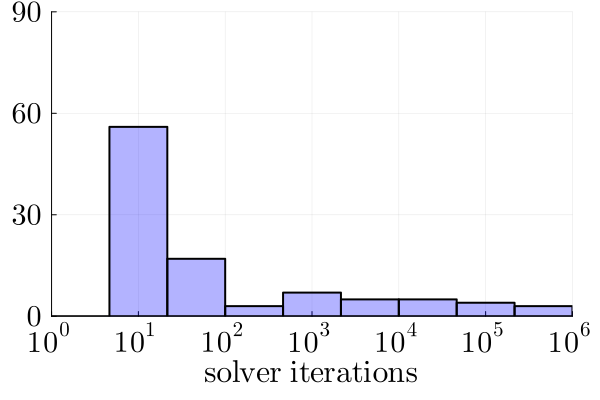}}
\caption{histograms of the number of solver iterations for $n\in\{8,9,10\}$. As $n$ increases, the number of instances for which the solver runs for a large number of iterations increases. Still, in most cases, a solution is found in about ten iterations.}
\label{fig_hist_iter}
\end{figure}

\begin{figure}
\centering
\includegraphics[width=.5\textwidth]{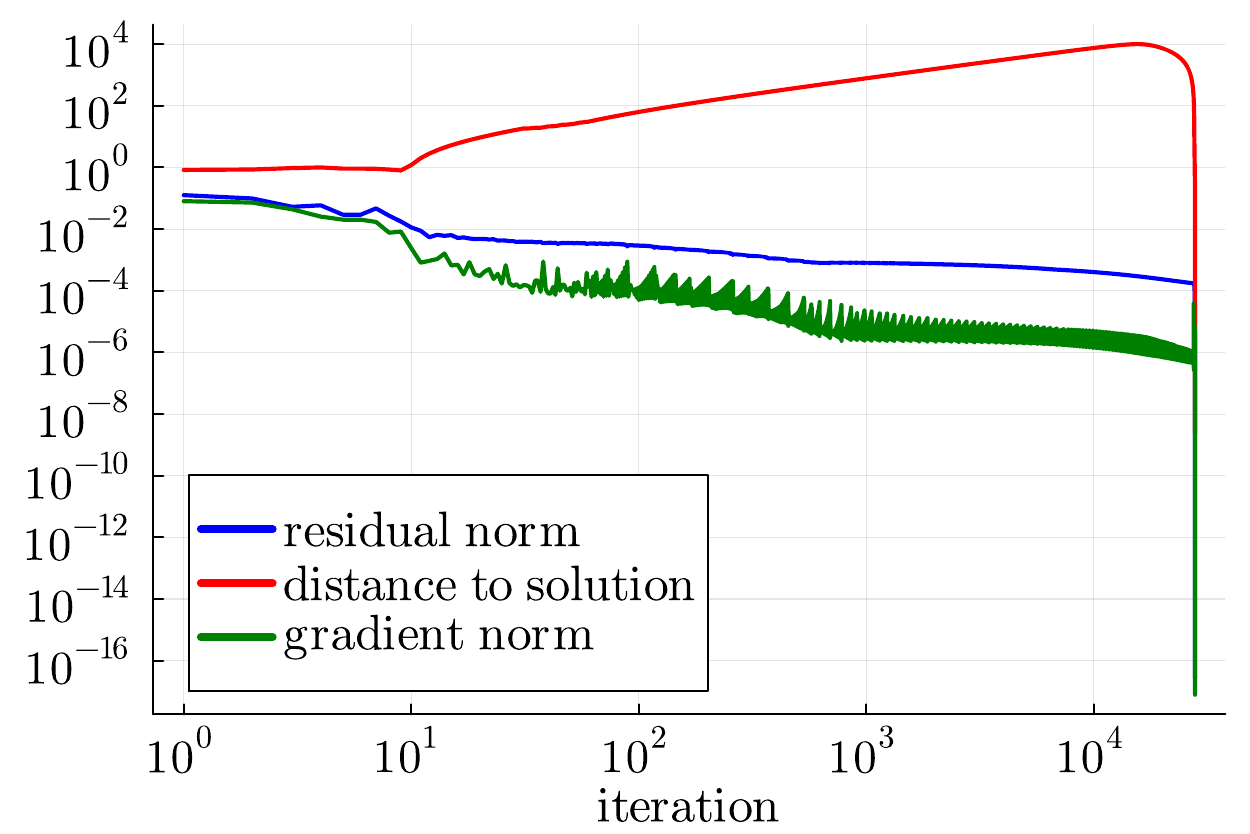}
\caption{convergence of the least squares solver (Powell's dog leg method) after a large number of iterations for an instance with $n=m=9$. The iterates reach a maximum distance to the unknown of $10^4$ while keeping a small value for the residual $\|\Lambda(\sigma)-\Lambda(\sigma^\dagger)\|_{\infty}$ (of the order of $10^{-4}$). Eventually, we reach a regime where the iterates quickly converge towards the true solution.}
\label{fig_conv_ls}
\end{figure}

\subsection{Comparison of least squares and Tikhonov regularization}
In this subsection, we investigate the impact of adding a squared $\ell^2$ norm penalization to the least squares objective (Tikhonov regularization). Namely, we solve
\begin{equation}
    \underset{\sigma\in[a,b]^n}{\mathrm{min}}~\|\Lambda(\sigma)-y\|_2^2+\lambda\|\sigma-[(a+b)/2]\mathbf{1}\|_2^2
    \tag{$\mathcal{Q}_{\lambda}(y)$}
    \label{tikhonov}
\end{equation}
for several values of $\lambda$. We stress that \eqref{tikhonov} is a non-convex optimization problem who could have spurious local minima. We are mainly interested in understanding whether the introduction of the regularization term could introduce additional spurious local minima.

We draw a large number of random unknown conductivities $\sigma^\dagger$ uniformly in $[a,b]^n$. Then, for each unknown conductivity, we draw a random guess uniformly in $[a,b]^n$ and run a trust-region algorithm initialized with this guess on \eqref{tikhonov} for several values of $\lambda$. The least squares estimate $\hat{\sigma}_{\mathrm{ls}}$ is defined as the estimated conducitivity obtained for $\lambda=0$ and the Tikhonov estimate $\hat{\sigma}_{\mathrm{tk}}$ as the estimated conductivity yielding the smallest $\ell^\infty$ estimation error among all tested positive values of $\lambda$.\footnote{This scenario favors Tikhonov regularization as the true $\ell^\infty$ estimation error would not be available in a practical scenario, making the choice of $\lambda$ more difficult.}

We take $n=3$ and $m=5$, as well as $(a,b)=(0.5, 1.5)$. The number of unknown conductivities is $100$ and the noise is a centered Gaussian with standard deviation $10^{-3}$. We consider values of $\lambda$ in the set $(10^{\alpha_i})_{1\leq i\leq 20}$ with $(\alpha_i)_{1\leq i\leq 20}$ a uniform discretization of $(-7,-2)$. The results of this experiment are displayed in \Cref{tik_ls}.  We notice that, on average, Tikhonov regularization yields a lower estimation error than least squares. Yet, for about half the instances, least squares performs slightly better and the optimal $\lambda$ is the smallest tested value (which is $10^{-7}$). Athough the addition of the squared $\ell^2$ norm term could in principle add new spurious local minima to the objective, we still observe that it is on average beneficial.

If the unknowns are small perturbations of the constant conductivity $[(a+b)/2]\mathbf{1}$, then Tikhonov regularization is even more beneficial. This is evidenced in \Cref{tik_ls_bis}, which displays the results of the same experiment as above, but with the unknowns being drawn uniformly at random in $[(a+b)/2 - \epsilon, (a+b)/2 + \epsilon]$ with $\epsilon = (b-a)/20$.

\begin{figure}
\centering
\subfloat[estimation error]{\includegraphics[width = 2.15in]{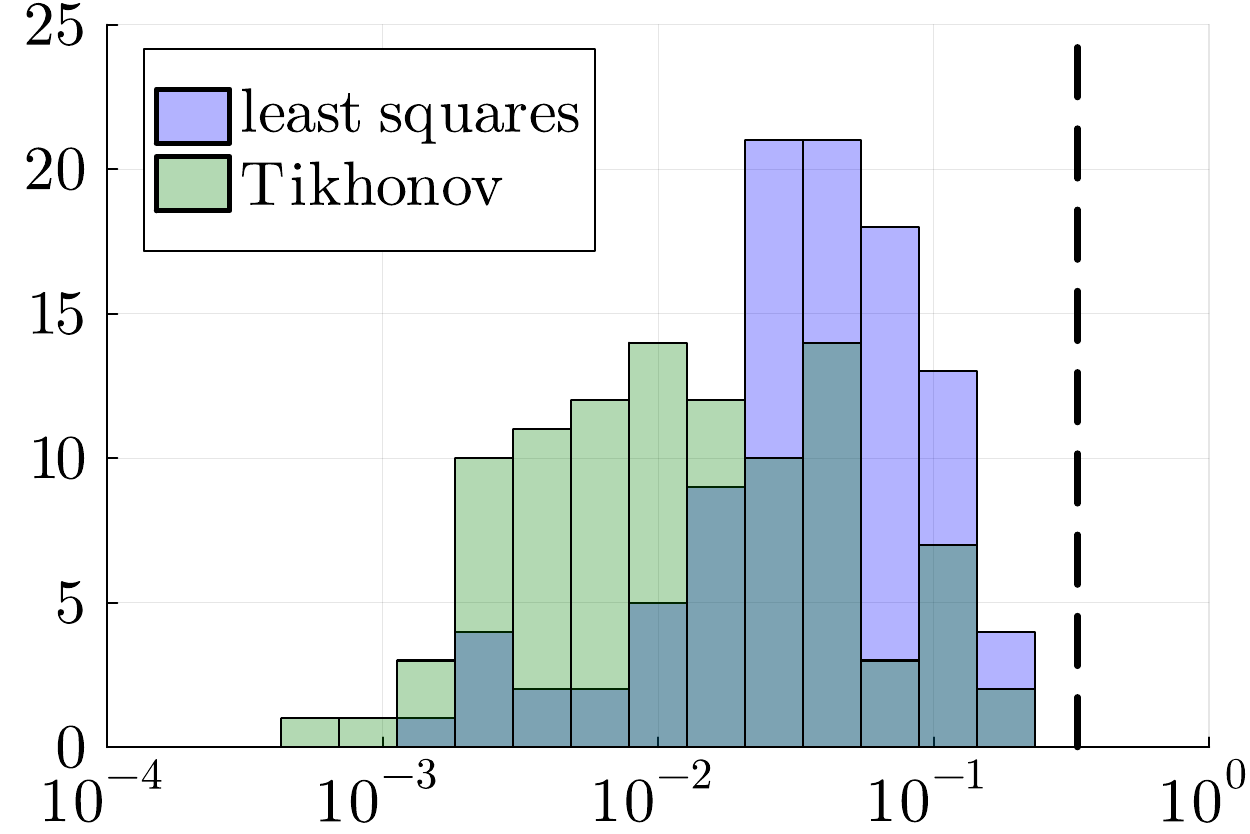}}
\subfloat[diff. of the estimation errors]{\includegraphics[width = 2.15in]{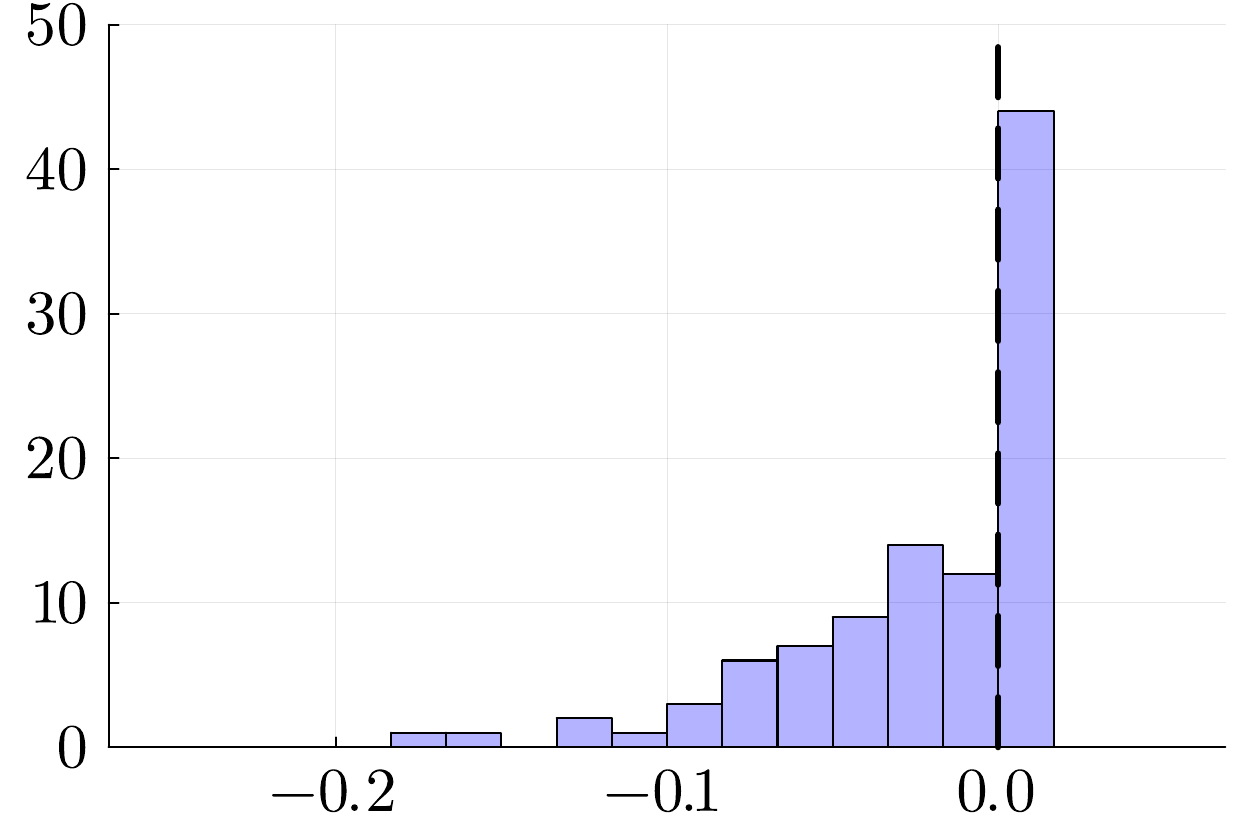}}
\subfloat[optimal $\lambda$]{\includegraphics[width = 2.15in]{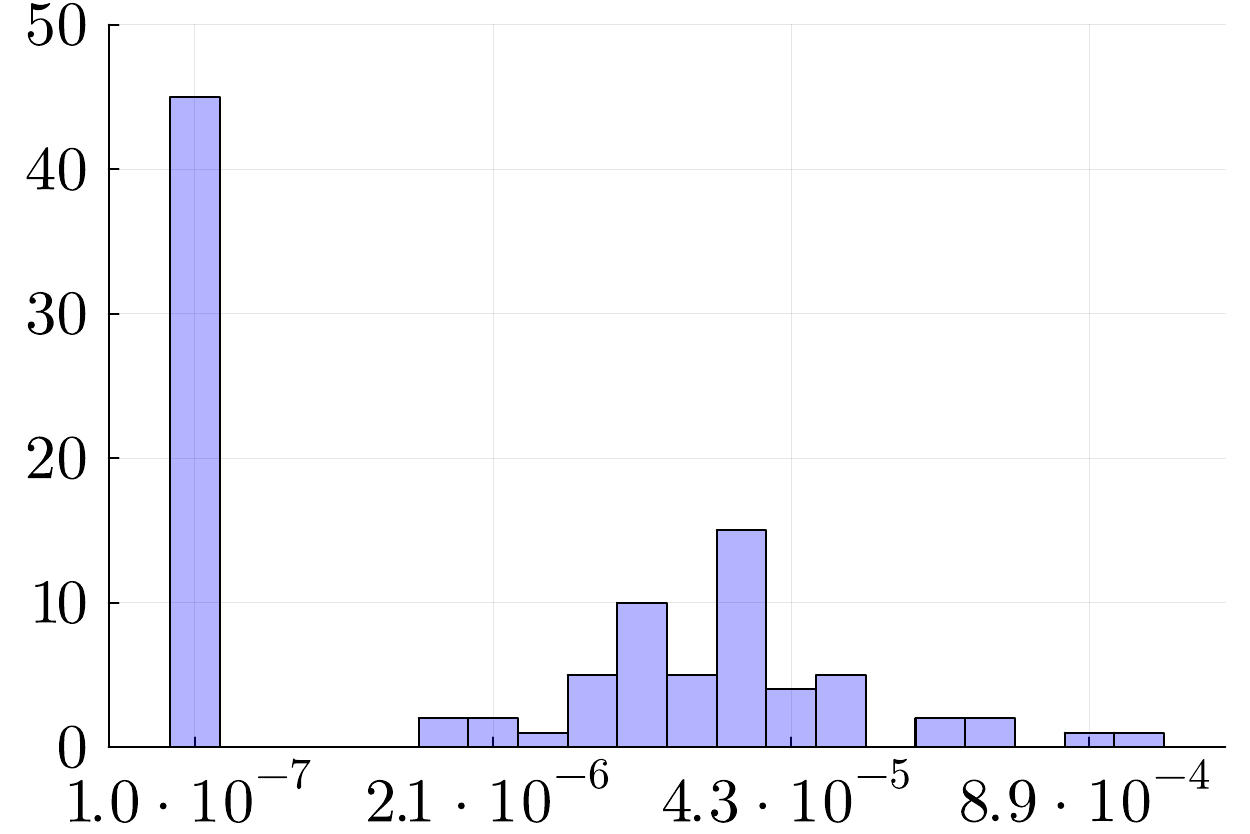}}
\caption{histograms of the estimation error $\|\hat{\sigma}-\sigma^\dagger\|_{\infty}$, of the difference $\|\hat{\sigma}_{\mathrm{tk}}-\sigma^\dagger\|_{\infty}-\|\hat{\sigma}_{\mathrm{ls}}-\sigma^\dagger\|_{\infty}$ of the Tikhonov estimation error and the least squares estimation error, and of the value of $\lambda$ yielding the smallest estimation error ($n=3$, $m=5$ and $(a,b)=(0.5, 1.5)$). The dashed black line in (a) corresponds to the average estimation error of a random guess (that is $(b-a)/3$).}
\label{tik_ls}
\end{figure}

\begin{figure}
\centering
\subfloat[estimation error]{\includegraphics[width = 2.15in]{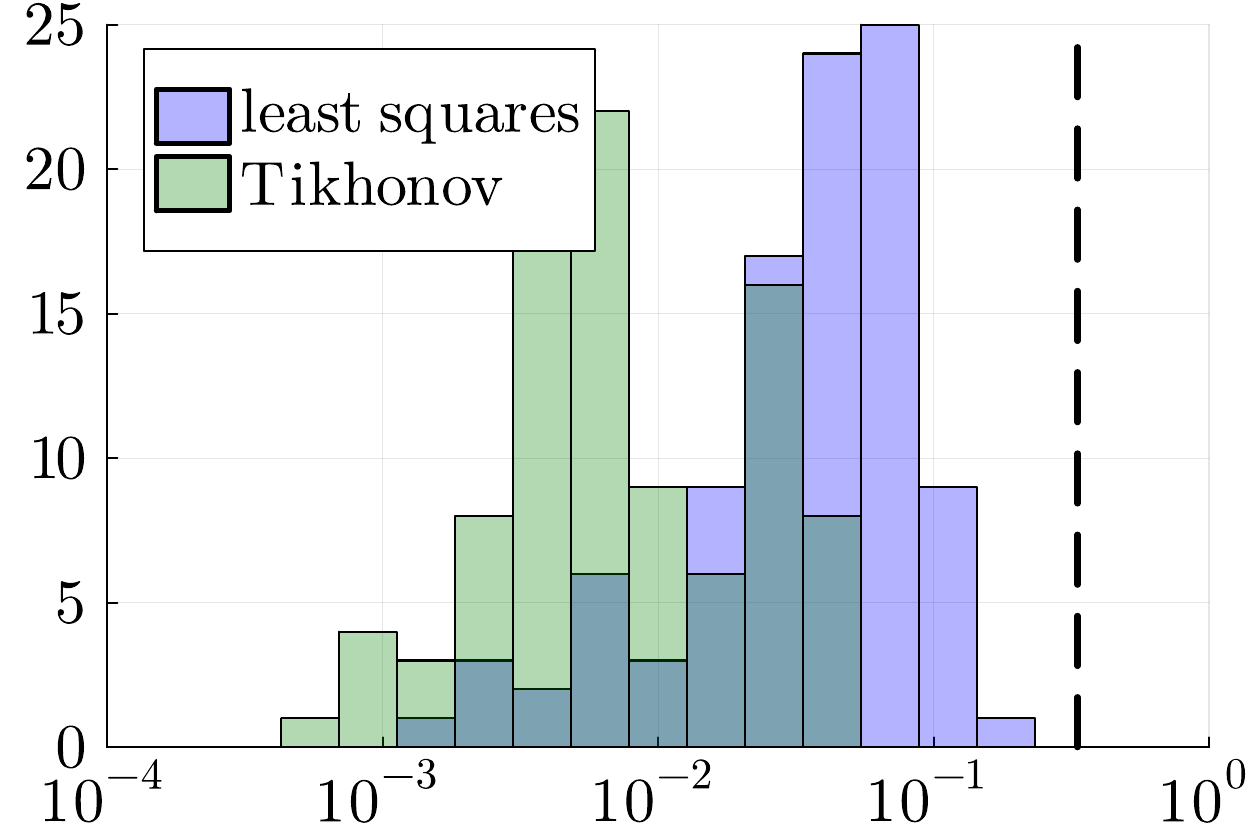}}
\subfloat[diff. of the estimation errors]{\includegraphics[width = 2.15in]{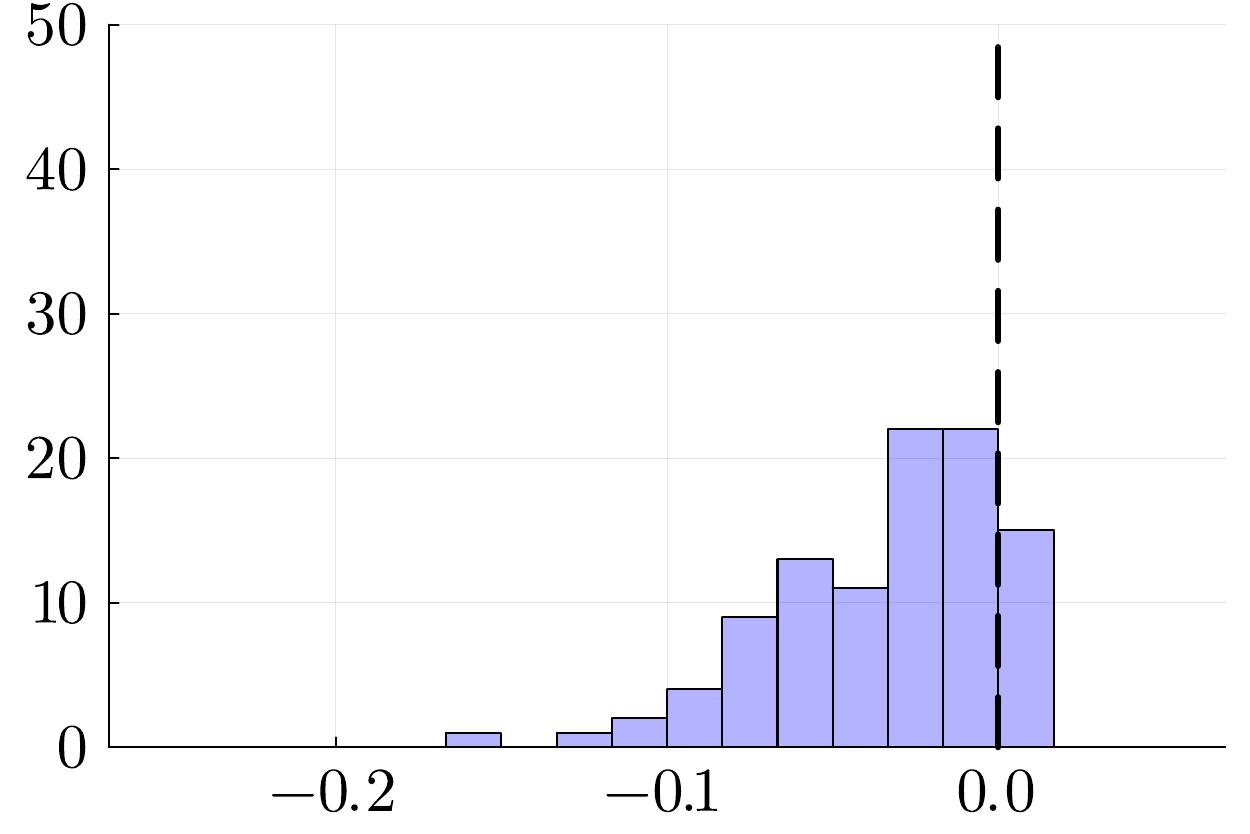}}
\subfloat[optimal $\lambda$]{\includegraphics[width = 2.15in]{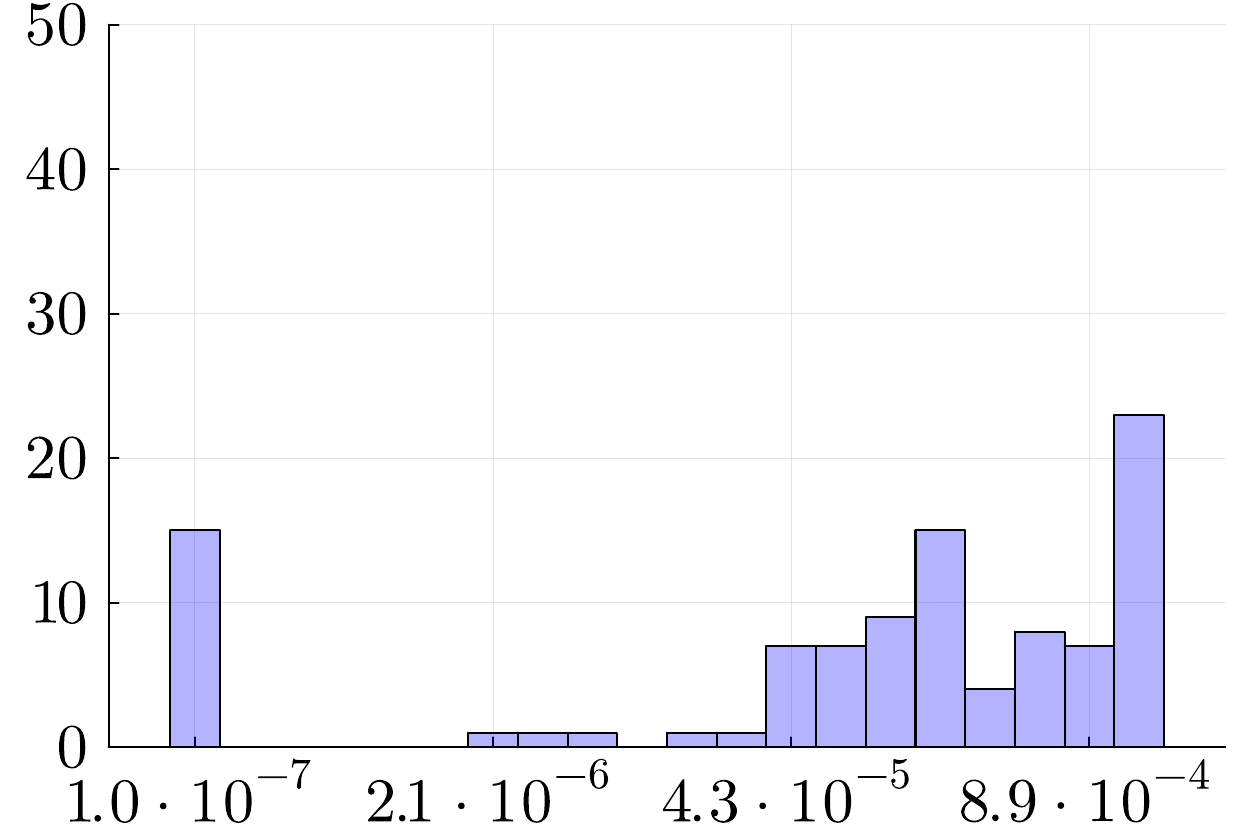}}
\caption{same as \Cref{tik_ls}, with unknowns being small perturbations of the constant conductivitiy ${[(a+b)/2]\mathbf{1}}$.}
\label{tik_ls_bis}
\end{figure}

\section{Conclusion}\label{sec:conclusion}
We carried out an extensive numerical investigation of the least squares approach in the context of the piecewise constant radial Calder\'on problem. Contrary to previous claims, this approach does not seem to suffer from the presence of spurious local minimums and we observed a good convergence behavior even in the presence of noise. Furthermore, we observed numerically that the Jacobian of the forward map has always full rank. Under this assumption {and an additional monotonicity condition (both rigorously verified in the case of $n=2$ unknowns)} we have proved that exact recovery is achieved whenever the number of measurements $m$ is at least the number of unknowns $n$. Under the same assumptions, we have showed that, when $m=n$, all critical points of the least squares objective are global minimums.

Our results suggest that the optimization landscape of the least squares objective warrants further investigations. An interesting and challenging future direction is the study of this landscape in the noisy regularized setting. The codebase we developed might also serve as a useful benchmark for further studies on this topic.

\section*{Acknowledgments}
G.S.A.\ is supported by the Air Force Office of Scientific Research under award number FA8655-23-1-7083 and by the European Union (ERC, SAMPDE, 101041040). Views and opinions expressed are however those of the authors only and do not necessarily reflect those of the European Union or the European Research Council. Neither the European Union nor the granting authority can be held responsible for them. G.S.A.\ is a member of the ``Gruppo Nazionale per l’Analisi Matematica, la Probabilità e le loro Applicazioni'', of the ``Istituto Nazionale di Alta Matematica''. The research was supported in part by the MIUR Excellence Department Project awarded to Dipartimento di Matematica, Università di Genova, CUP D33C23001110001. Co-funded by European Union – Next Generation EU, Missione 4 Componente 1 CUP D53D23005770006 and CUP D53D23016180001. The work of R.P. was supported by the European Union (ERC, SAMPDE, 101041040 and ERC, WOLF, 101141361).
G.S.A. and C.P. acknowledge  support from a Royal Society International Exchanges award. 

\bibliography{ref}
\bibliographystyle{apalike}
\end{document}

%% file: plots/fig_radial.tex
\tikzset{every picture/.style={line width=0.75pt}} 

\begin{tikzpicture}[x=0.75pt,y=0.75pt,yscale=-1,xscale=1]
\useasboundingbox (140,20) rectangle (440,260);

\draw  [fill={rgb, 255:red, 255; green, 0; blue, 0 }  ,fill opacity=1 ] (180.5,140) .. controls (180.5,79.52) and (229.52,30.5) .. (290,30.5) .. controls (350.48,30.5) and (399.5,79.52) .. (399.5,140) .. controls (399.5,200.48) and (350.48,249.5) .. (290,249.5) .. controls (229.52,249.5) and (180.5,200.48) .. (180.5,140) -- cycle ;
\draw  [fill={rgb, 255:red, 255; green, 114; blue, 0 }  ,fill opacity=1 ] (208.5,139.25) .. controls (208.5,94.65) and (244.65,58.5) .. (289.25,58.5) .. controls (333.85,58.5) and (370,94.65) .. (370,139.25) .. controls (370,183.85) and (333.85,220) .. (289.25,220) .. controls (244.65,220) and (208.5,183.85) .. (208.5,139.25) -- cycle ;
\draw  [fill={rgb, 255:red, 255; green, 176; blue, 0 }  ,fill opacity=1 ] (240.14,140) .. controls (240.14,112.46) and (262.46,90.14) .. (290,90.14) .. controls (317.54,90.14) and (339.86,112.46) .. (339.86,140) .. controls (339.86,167.54) and (317.54,189.86) .. (290,189.86) .. controls (262.46,189.86) and (240.14,167.54) .. (240.14,140) -- cycle ;
\draw  [fill={rgb, 255:red, 255; green, 247; blue, 0 }  ,fill opacity=1 ] (270.16,140) .. controls (270.16,129.04) and (279.04,120.16) .. (290,120.16) .. controls (300.96,120.16) and (309.84,129.04) .. (309.84,140) .. controls (309.84,150.96) and (300.96,159.84) .. (290,159.84) .. controls (279.04,159.84) and (270.16,150.96) .. (270.16,140) -- cycle ;
\draw    (300.5,156.5) -- (307.64,164) ;
\draw    (330.5,169) -- (337.64,175) ;
\draw    (360,179.5) -- (366.14,185) ;
\draw    (387,191) -- (394.14,197) ;
\draw    (295.5,136) -- (389.14,78.5) ;

\draw (270,41) node [anchor=north west][inner sep=0.75pt]    {$\sigma =\sigma _{1}$};
\draw (270,72) node [anchor=north west][inner sep=0.75pt]    {$\sigma =\sigma _{2}$};
\draw (270,104) node [anchor=north west][inner sep=0.75pt]    {$\sigma =\sigma _{3}$};
\draw (392,70) node [anchor=north west][inner sep=0.75pt]    {$\sigma =\sigma _{4}$};
\draw (395.5,192) node [anchor=north west][inner sep=0.75pt]    {$r_{0} =1$};
\draw (367.5,185) node [anchor=north west][inner sep=0.75pt]    {$r_{1}$};
\draw (339,175) node [anchor=north west][inner sep=0.75pt]    {$r_{2}$};
\draw (307.84,165) node [anchor=north west][inner sep=0.75pt]    {$r_{3}$};

\end{tikzpicture}

%% file: ref.bib
@article{klibanov2025convexification,
  title={Convexification with the viscocity term for electrical impedance tomography},
  author={Klibanov, Michael V and Li, Jingzhi and Yang, Zhipeng},
  journal={arXiv preprint arXiv:2503.07916},
  year={2025}
}

@article{albertiInfiniteDimensionalInverseProblems2022,
  title = {Infinite-{{Dimensional Inverse Problems}} with {{Finite Measurements}}},
  author = {Alberti, Giovanni S. and Santacesaria, Matteo},
  year = {2022},
  month = jan,
  journal = {Archive for Rational Mechanics and Analysis},
  volume = {243},
  number = {1},
  pages = {1--31},
  issn = {1432-0673},
  doi = {10.1007/s00205-021-01718-4},
  urldate = {2024-01-24},
  langid = {english}
}

@article {alessandrini-scapin-2017,
    AUTHOR = {Alessandrini, Giovanni and Scapin, Andrea},
     TITLE = {Depth dependent resolution in electrical impedance tomography},
   JOURNAL = {J. Inverse Ill-Posed Probl.},
  FJOURNAL = {Journal of Inverse and Ill-Posed Problems},
    VOLUME = {25},
      YEAR = {2017},
    NUMBER = {3},
     PAGES = {391--402},
      ISSN = {0928-0219,1569-3945},
   MRCLASS = {65N21 (30C35 35J25 35R25 35R30 78A45 92C55)},
  MRNUMBER = {3652265},
       DOI = {10.1515/jiip-2017-0029},
       URL = {https://doi.org/10.1515/jiip-2017-0029},
}

@article {winkler-rieder-2014,
    AUTHOR = {Winkler, Robert and Rieder, Andreas},
     TITLE = {Resolution-controlled conductivity discretization in
              electrical impedance tomography},
   JOURNAL = {SIAM J. Imaging Sci.},
  FJOURNAL = {SIAM Journal on Imaging Sciences},
    VOLUME = {7},
      YEAR = {2014},
    NUMBER = {4},
     PAGES = {2048--2077},
      ISSN = {1936-4954},
   MRCLASS = {65N20 (65F22 92C55)},
  MRNUMBER = {3268610},
MRREVIEWER = {Hans-Peter\ Helfrich},
       DOI = {10.1137/140958955},
       URL = {https://doi.org/10.1137/140958955},
}

@article {garde-nyvonen-2020,
    AUTHOR = {Garde, Henrik and Hyv\"onen, Nuutti},
     TITLE = {Optimal depth-dependent distinguishability bounds for
              electrical impedance tomography in arbitrary dimension},
   JOURNAL = {SIAM J. Appl. Math.},
  FJOURNAL = {SIAM Journal on Applied Mathematics},
    VOLUME = {80},
      YEAR = {2020},
    NUMBER = {1},
     PAGES = {20--43},
      ISSN = {0036-1399,1095-712X},
   MRCLASS = {35R30 (35J05 35P15)},
  MRNUMBER = {4046786},
MRREVIEWER = {Jonathan\ Rohleder},
       DOI = {10.1137/19M1258761},
       URL = {https://doi.org/10.1137/19M1258761},
}

@article {garde2020,
    AUTHOR = {Garde, Henrik},
     TITLE = {Reconstruction of piecewise constant layered conductivities in
              electrical impedance tomography},
   JOURNAL = {Comm. Partial Differential Equations},
  FJOURNAL = {Communications in Partial Differential Equations},
    VOLUME = {45},
      YEAR = {2020},
    NUMBER = {9},
     PAGES = {1118--1133},
      ISSN = {0360-5302,1532-4133},
   MRCLASS = {35R30 (35Q60 35R05)},
  MRNUMBER = {4134387},
MRREVIEWER = {Michiyuki\ Watanabe},
       DOI = {10.1080/03605302.2020.1760884},
       URL = {https://doi.org/10.1080/03605302.2020.1760884},
}

@article {garde-2022,
    AUTHOR = {Garde, Henrik},
     TITLE = {Simplified reconstruction of layered materials in {EIT}},
   JOURNAL = {Appl. Math. Lett.},
  FJOURNAL = {Applied Mathematics Letters. An International Journal of Rapid
              Publication},
    VOLUME = {126},
      YEAR = {2022},
     PAGES = {Paper No. 107815, 5},
      ISSN = {0893-9659,1873-5452},
   MRCLASS = {35Q60 (78A46)},
  MRNUMBER = {4349266},
       DOI = {10.1016/j.aml.2021.107815},
       URL = {https://doi.org/10.1016/j.aml.2021.107815},
}

@article {garde-knudsen-2017,
    AUTHOR = {Garde, Henrik and Knudsen, Kim},
     TITLE = {Distinguishability revisited: depth dependent bounds on
              reconstruction quality in electrical impedance tomography},
   JOURNAL = {SIAM J. Appl. Math.},
  FJOURNAL = {SIAM Journal on Applied Mathematics},
    VOLUME = {77},
      YEAR = {2017},
    NUMBER = {2},
     PAGES = {697--720},
      ISSN = {0036-1399,1095-712X},
   MRCLASS = {35R30 (35J25 35P15 35R05 35R25)},
  MRNUMBER = {3639593},
MRREVIEWER = {Viatcheslav\ I.\ Pri\u imenko},
       DOI = {10.1137/16M1072991},
       URL = {https://doi.org/10.1137/16M1072991},
}

@article{wachterImplementationInteriorpointFilter2006,
  title = {On the Implementation of an Interior-Point Filter Line-Search Algorithm for Large-Scale Nonlinear Programming},
  author = {W{\"a}chter, Andreas and Biegler, Lorenz T.},
  year = {2006},
  month = mar,
  journal = {Mathematical Programming},
  volume = {106},
  number = {1},
  pages = {25--57},
  issn = {1436-4646},
  doi = {10.1007/s10107-004-0559-y},
  urldate = {2025-06-19},
  langid = {english}
}

@article {rondi-2006,
    AUTHOR = {Rondi, Luca},
     TITLE = {A remark on a paper by {G}. {A}lessandrini and {S}.
              {V}essella: ``{L}ipschitz stability for the inverse
              conductivity problem'' [{A}dv. in {A}ppl. {M}ath. {\bf 35}
              (2005), no. 2, 207--241; MR2152888]},
   JOURNAL = {Adv. in Appl. Math.},
  FJOURNAL = {Advances in Applied Mathematics},
    VOLUME = {36},
      YEAR = {2006},
    NUMBER = {1},
     PAGES = {67--69},
      ISSN = {0196-8858,1090-2074},
   MRCLASS = {35R30 (35B35)},
  MRNUMBER = {2198854},
       DOI = {10.1016/j.aam.2004.12.003},
       URL = {https://doi.org/10.1016/j.aam.2004.12.003},
}

@article {dicristo-rondi-2021,
    AUTHOR = {Di Cristo, Michele and Rondi, Luca},
     TITLE = {Interior decay of solutions to elliptic equations with respect
              to frequencies at the boundary},
   JOURNAL = {Indiana Univ. Math. J.},
  FJOURNAL = {Indiana University Mathematics Journal},
    VOLUME = {70},
      YEAR = {2021},
    NUMBER = {4},
     PAGES = {1303--1334},
      ISSN = {0022-2518,1943-5258},
   MRCLASS = {35J30},
  MRNUMBER = {4318476},
       DOI = {10.1512/iumj.2021.70.9367},
       URL = {https://doi.org/10.1512/iumj.2021.70.9367},
}

@article {garofalo-lin-1986,
    AUTHOR = {Garofalo, Nicola and Lin, Fang-Hua},
     TITLE = {Monotonicity properties of variational integrals, {$A_p$}
              weights and unique continuation},
   JOURNAL = {Indiana Univ. Math. J.},
  FJOURNAL = {Indiana University Mathematics Journal},
    VOLUME = {35},
      YEAR = {1986},
    NUMBER = {2},
     PAGES = {245--268},
      ISSN = {0022-2518,1943-5258},
   MRCLASS = {35J20 (35J10 42B25)},
  MRNUMBER = {833393},
MRREVIEWER = {Stavros\ A.\ Belbas},
       DOI = {10.1512/iumj.1986.35.35015},
       URL = {https://doi.org/10.1512/iumj.1986.35.35015},
}

@article{knudsenDBarMethodElectrical2007,
  title = {D-{{Bar Method}} for {{Electrical Impedance Tomography}} with {{Discontinuous Conductivities}}},
  author = {Knudsen, Kim and Lassas, Matti and Mueller, Jennifer L. and Siltanen, Samuli},
  year = {2007},
  month = jan,
  journal = {SIAM Journal on Applied Mathematics},
  volume = {67},
  number = {3},
  pages = {893--913},
  publisher = {{Society for Industrial and Applied Mathematics}},
  issn = {0036-1399},
  doi = {10.1137/060656930},
  urldate = {2025-06-02}
}

@article{siltanenImplementationReconstructionAlgorithm2000,
  title = {An Implementation of the Reconstruction Algorithm of {{A Nachman}} for the {{2D}} Inverse Conductivity Problem},
  author = {Siltanen, Samuli and Mueller, Jennifer and Isaacson, David},
  year = {2000},
  month = jun,
  journal = {Inverse Problems},
  volume = {16},
  number = {3},
  pages = {681},
  issn = {0266-5611},
  doi = {10.1088/0266-5611/16/3/310},
  urldate = {2025-06-02},
  langid = {english}
}

@article{barceloBornApproximationThreedimensional2024,
  title = {The {{Born}} Approximation in the Three-Dimensional {{Calder{\'o}n}} Problem II: {{Numerical}} Reconstruction in the Radial Case},
  shorttitle = {The {{Born}} Approximation in the Three-Dimensional {{Calder{\'o}n}} Problem II},
  author = {Barcel{\'o}, Juan A. and Castro, Carlos and Maci{\`a}, Fabricio and Mero{\~n}o, Crist{\'o}bal J.},
  year = {Thu Feb 01 05:00:00 UTC 2024},
  journal = {Inverse Problems and Imaging},
  volume = {18},
  number = {1},
  pages = {183--207},
  publisher = {{Inverse Problems and Imaging}},
  issn = {1930-8337},
  doi = {10.3934/ipi.2023029},
  urldate = {2025-06-02},
  copyright = {http://creativecommons.org/licenses/by/3.0/},
  langid = {english}
}

@misc{daudeStableFactorizationCalderon2024,
  title = {Stable Factorization of the {{Calder{\'o}n}} Problem via the {{Born}} Approximation},
  author = {Daud{\'e}, Thierry and Maci{\`a}, Fabricio and Mero{\~n}o, Crist{\'o}bal J. and Nicoleau, Fran{\c c}ois},
  year = {2024},
  month = feb,
  number = {arXiv:2402.06321},
  eprint = {2402.06321},
  primaryclass = {math},
  publisher = {arXiv},
  doi = {10.48550/arXiv.2402.06321},
  urldate = {2025-06-02},
  archiveprefix = {arXiv}
}

@article{lazzaro2024oracle,
  title={Oracle-Net for nonlinear compressed sensing in Electrical Impedance Tomography reconstruction problems},
  author={Lazzaro, Damiana and Morigi, Serena and Ratti, Luca},
  journal={Journal of Scientific Computing},
  volume={101},
  number={2},
  pages={49},
  year={2024},
  publisher={Springer}
}

@article{harrach2025monotonicity,
  title={A monotonicity-based globalization of the level-set method for inclusion detection},
  author={Harrach, Bastian and Meftahi, Houcine},
  journal={arXiv preprint arXiv:2501.15887},
  year={2025}
}

@article{brojatsch2024required,
  title={On the required number of electrodes for uniqueness and convex reformulation in an inverse coefficient problem},
  author={Brojatsch, Andrej and Harrach, Bastian},
  journal={arXiv preprint arXiv:2411.00482},
  year={2024}
}

@article{klibanov2017globally,
  title={Globally strictly convex cost functional for a 1-D inverse medium scattering problem with experimental data},
  author={Klibanov, Michael V and Kolesov, Aleksandr E and Nguyen, Lam and Sullivan, Anders},
  journal={SIAM Journal on Applied Mathematics},
  volume={77},
  number={5},
  pages={1733--1755},
  year={2017},
  publisher={SIAM}
}

@article{sylvesterConvergentLayerStripping1992,
  title = {A {{Convergent Layer Stripping Algorithm}} for the {{Radially Symmetric Impedence Tomography Problem}}},
  author = {Sylvester, John},
  year = {1992},
  month = jan,
  journal = {Communications in Partial Differential Equations},
  volume = {17},
  number = {11-12},
  pages = {1955--1994},
  publisher = {Taylor \& Francis},
  issn = {0360-5302},
  doi = {10.1080/03605309208820910},
  urldate = {2025-05-19}
}

@misc{revelsForwardModeAutomaticDifferentiation2016,
  title = {Forward-{{Mode Automatic Differentiation}} in {{Julia}}},
  author = {Revels, Jarrett and Lubin, Miles and Papamarkou, Theodore},
  year = {2016},
  month = jul,
  number = {arXiv:1607.07892},
  eprint = {1607.07892},
  primaryclass = {cs},
  publisher = {arXiv},
  doi = {10.48550/arXiv.1607.07892},
  urldate = {2025-05-19},
  archiveprefix = {arXiv}
}

@article{kaltenbacher-2024,
    author = {Kaltenbacher, Barbara},
    title = {Convergence rates under a range invariance condition with application to electrical impedance tomography},
    journal = {IMA Journal of Numerical Analysis},
    pages = {drae063},
    year = {2024},
    month = {09},
    issn = {0272-4979},
    doi = {10.1093/imanum/drae063},
    url = {https://doi.org/10.1093/imanum/drae063},
    eprint = {https://academic.oup.com/imajna/advance-article-pdf/doi/10.1093/imanum/drae063/59381393/drae063.pdf},
}

@book {kaltenbacher-neubauer-scherzer-2008,
    AUTHOR = {Kaltenbacher, Barbara and Neubauer, Andreas and Scherzer,
              Otmar},
     TITLE = {Iterative regularization methods for nonlinear ill-posed
              problems},
    SERIES = {Radon Series on Computational and Applied Mathematics},
    VOLUME = {6},
 PUBLISHER = {Walter de Gruyter GmbH \& Co. KG, Berlin},
      YEAR = {2008},
     PAGES = {viii+194},
      ISBN = {978-3-11-020420-9},
   MRCLASS = {65-02 (47J25 65J22 65R32)},
  MRNUMBER = {2459012},
MRREVIEWER = {Thorsten\ Hohage},
       DOI = {10.1515/9783110208276},
       URL = {https://doi.org/10.1515/9783110208276},
}

@article {Lechleiter-Rieder-2008,
    AUTHOR = {Lechleiter, Armin and Rieder, Andreas},
     TITLE = {Newton regularizations for impedance tomography: convergence
              by local injectivity},
   JOURNAL = {Inverse Problems},
  FJOURNAL = {Inverse Problems. An International Journal on the Theory and
              Practice of Inverse Problems, Inverse Methods and Computerized
              Inversion of Data},
    VOLUME = {24},
      YEAR = {2008},
    NUMBER = {6},
     PAGES = {065009, 18},
      ISSN = {0266-5611,1361-6420},
   MRCLASS = {78A70 (35J25 65N21)},
  MRNUMBER = {2456956},
       DOI = {10.1088/0266-5611/24/6/065009},
       URL = {https://doi.org/10.1088/0266-5611/24/6/065009},
}

@article{bacchelliLipschitzStabilityStationary2006,
  title = {Lipschitz Stability for a Stationary {{2D}} Inverse Problem with Unknown Polygonal Boundary},
  author = {Bacchelli, Valeria and Vessella, Sergio},
  year = {2006},
  month = aug,
  journal = {Inverse Problems},
  volume = {22},
  number = {5},
  pages = {1627},
  issn = {0266-5611},
  doi = {10.1088/0266-5611/22/5/007},
  urldate = {2025-05-19},
  langid = {english}
}

@article{berettaLipschitzStabilityInverse2013,
  title = {Lipschitz {{Stability}} of an {{Inverse Boundary Value Problem}} for a {{Schr{\"o}dinger-Type Equation}}},
  author = {Beretta, Elena and {de Hoop}, Maarten V. and Qiu, Lingyun},
  year = {2013},
  month = jan,
  journal = {SIAM Journal on Mathematical Analysis},
  volume = {45},
  number = {2},
  pages = {679--699},
  publisher = {{Society for Industrial and Applied Mathematics}},
  issn = {0036-1410},
  doi = {10.1137/120869201},
  urldate = {2025-05-19}
}

@article{berettaGlobalLipschitzStability2022,
  title = {Global {{Lipschitz}} Stability Estimates for Polygonal Conductivity Inclusions from Boundary Measurements},
  author = {Beretta, Elena and Francini, Elisa},
  year = {2022},
  month = jul,
  journal = {Applicable Analysis},
  volume = {101},
  number = {10},
  pages = {3536--3549},
  publisher = {Taylor \& Francis},
  issn = {0003-6811},
  doi = {10.1080/00036811.2020.1775819},
  urldate = {2025-05-19}
}

@article{alessandriniLipschitzStabilityInverse2005,
  title = {Lipschitz Stability for the Inverse Conductivity Problem},
  author = {Alessandrini, Giovanni and Vessella, Sergio},
  year = {2005},
  month = aug,
  journal = {Advances in Applied Mathematics},
  volume = {35},
  number = {2},
  pages = {207--241},
  issn = {0196-8858},
  doi = {10.1016/j.aam.2004.12.002},
  urldate = {2025-05-19}
}

@article{mandacheExponentialInstabilityInverse2001,
  title = {Exponential Instability in an Inverse Problem for the {{Schr{\"o}dinger}} Equation},
  author = {Mandache, Niculae},
  year = {2001},
  month = aug,
  journal = {Inverse Problems},
  volume = {17},
  number = {5},
  pages = {1435},
  issn = {0266-5611},
  doi = {10.1088/0266-5611/17/5/313},
  urldate = {2025-05-19},
  langid = {english}
}

@article{alessandriniStableDeterminationConductivity1988,
  title = {Stable Determination of Conductivity by Boundary Measurements},
  author = {Alessandrini, Giovanni},
  year = {1988},
  month = jan,
  journal = {Applicable Analysis},
  volume = {27},
  number = {1-3},
  pages = {153--172},
  publisher = {Taylor \& Francis},
  issn = {0003-6811},
  doi = {10.1080/00036818808839730},
  urldate = {2025-05-19}
}

@article{alessandriniStrongUniqueContinuation2012,
  title = {Strong Unique Continuation for General Elliptic Equations in {{2D}}},
  author = {Alessandrini, Giovanni},
  year = {2012},
  month = feb,
  journal = {Journal of Mathematical Analysis and Applications},
  volume = {386},
  number = {2},
  pages = {669--676},
  issn = {0022-247X},
  doi = {10.1016/j.jmaa.2011.08.029},
  urldate = {2025-05-14}
}

@article{neubauerLandweberIterationNonlinear2000,
  title = {On {{Landweber}} Iteration for Nonlinear Ill-Posed Problems in {{Hilbert}} Scales},
  author = {Neubauer, Andreas},
  year = {2000},
  month = apr,
  journal = {Numerische Mathematik},
  volume = {85},
  number = {2},
  pages = {309--328},
  issn = {0945-3245},
  doi = {10.1007/s002110050487},
  urldate = {2025-04-29},
  langid = {english}
}

@article{hankeConvergenceAnalysisLandweber1995,
  title = {A Convergence Analysis of the {{Landweber}} Iteration for Nonlinear Ill-Posed Problems},
  author = {Hanke, Martin and Neubauer, Andreas and Scherzer, Otmar},
  year = {1995},
  month = nov,
  journal = {Numerische Mathematik},
  volume = {72},
  number = {1},
  pages = {21--37},
  issn = {0945-3245},
  doi = {10.1007/s002110050158},
  urldate = {2025-04-29},
  langid = {english}
}

@article {kindermann-2022,
    AUTHOR = {Kindermann, Stefan},
     TITLE = {On the tangential cone condition for electrical impedance
              tomography},
   JOURNAL = {Electron. Trans. Numer. Anal.},
  FJOURNAL = {Electronic Transactions on Numerical Analysis},
    VOLUME = {57},
      YEAR = {2022},
     PAGES = {17--34},
      ISSN = {1068-9613},
   MRCLASS = {65N21 (47J05 65F22)},
  MRNUMBER = {4418424},
MRREVIEWER = {Thomas\ Schuster},
       DOI = {10.1553/etna\_vol57s17},
       URL = {https://doi.org/10.1553/etna_vol57s17},
}

@article{albertiCalderonsInverseProblem2019,
  title = {Calder{\'o}n's Inverse Problem with a Finite Number of Measurements},
  author = {Alberti, Giovanni S. and Santacesaria, Matteo},
  year = {2019},
  month = jan,
  journal = {Forum of Mathematics, Sigma},
  volume = {7},
  pages = {e35},
  issn = {2050-5094},
  doi = {10.1017/fms.2019.31},
  urldate = {2024-11-18},
  langid = {english}
}

@article{sylvesterGlobalUniquenessTheorem1987,
  title = {A {{Global Uniqueness Theorem}} for an {{Inverse Boundary Value Problem}}},
  author = {Sylvester, John and Uhlmann, Gunther},
  year = {1987},
  journal = {Annals of Mathematics},
  volume = {125},
  number = {1},
  eprint = {1971291},
  eprinttype = {jstor},
  pages = {153--169},
  publisher = {[Annals of Mathematics, Trustees of Princeton University on Behalf of the Annals of Mathematics, Mathematics Department, Princeton University]},
  issn = {0003-486X},
  doi = {10.2307/1971291},
  urldate = {2025-05-11}
}

@article{nachmanGlobalUniquenessTwoDimensional1996,
  title = {Global {{Uniqueness}} for a {{Two-Dimensional Inverse Boundary Value Problem}}},
  author = {Nachman, Adrian I.},
  year = {1996},
  journal = {Annals of Mathematics},
  volume = {143},
  number = {1},
  eprint = {2118653},
  eprinttype = {jstor},
  pages = {71--96},
  publisher = {[Annals of Mathematics, Trustees of Princeton University on Behalf of the Annals of Mathematics, Mathematics Department, Princeton University]},
  issn = {0003-486X},
  doi = {10.2307/2118653},
  urldate = {2025-05-11}
}

@article{caroGlobalUniquenessCalderon2016,
  title = {Global {{Uniqueness}} for the {{Calder{\'o}n Problem}} with {{Lipschitz Conductivities}}},
  author = {Caro, Pedro and Rogers, Keith M.},
  year = {2016},
  month = jan,
  journal = {Forum of Mathematics, Pi},
  volume = {4},
  pages = {e2},
  issn = {2050-5086},
  doi = {10.1017/fmp.2015.9},
  urldate = {2025-05-11},
  langid = {english}
}

@article{bukhgeimRecoveringPotentialCauchy2008,
  title = {Recovering a Potential from {{Cauchy}} Data in the Two-Dimensional Case},
  author = {Bukhgeim, A. L.},
  year = {2008},
  month = jan,
  journal = {Journal of Inverse and Ill-posed Problems},
  volume = {16},
  number = {1},
  pages = {19--33},
  publisher = {De Gruyter},
  issn = {1569-3945},
  doi = {10.1515/jiip.2008.002},
  urldate = {2025-05-11},
  chapter = {Journal of Inverse and Ill-posed Problems},
  copyright = {De Gruyter expressly reserves the right to use all content for commercial text and data mining within the meaning of Section 44b of the German Copyright Act.},
  langid = {english}
}

@article{astala2006calderon,
  title={Calder{\'o}n's inverse conductivity problem in the plane},
  author={Astala, Kari and P{\"a}iv{\"a}rinta, Lassi},
  journal={Annals of Mathematics},
  pages={265--299},
  year={2006},
  publisher={JSTOR}
}

@article{klibanovConvexificationElectricalImpedance2019,
  title = {Convexification of Electrical Impedance Tomography with Restricted {{Dirichlet-to-Neumann}} Map Data},
  author = {Klibanov, Michael V and Li, Jingzhi and Zhang, Wenlong},
  year = {2019},
  month = feb,
  journal = {Inverse Problems},
  volume = {35},
  number = {3},
  pages = {035005},
  publisher = {IOP Publishing},
  issn = {0266-5611},
  doi = {10.1088/1361-6420/aafecd},
  urldate = {2025-01-22},
  langid = {english}
}

@article{calderon_inverse_1980,
	title = {On an inverse boundary value problem},
	volume = {25},
	issn = {1807-0302},
	doi = {10.1590/S0101-82052006000200002},
	journal = {Seminar on Numerical Analysis and its Applications to Continuum Physics},
	author = {Calderón, Alberto P},
	year = {1980},
	pages = {65--73},
}

@misc{Clarabel_2024,
      title={Clarabel: An interior-point solver for conic programs with quadratic objectives}, 
      author={Paul J. Goulart and Yuwen Chen},
      year={2024},
      eprint={2405.12762},
      archivePrefix={arXiv},
      primaryClass={math.OC}
}

@article{albertiConvexLiftingApproach2025,
  title = {A Convex Lifting Approach for the {{Calder{\'o}n}} Problem},
  author = {Alberti, Giovanni S. and Petit, Romain and Sanna, Simone},
  year = {2025},
  journal={arXiv preprint arXiv:2507.00645}
}

@article{ocpb:16,
    author       = {Brendan O'Donoghue and Eric Chu and Neal Parikh and Stephen Boyd},
    title        = {Conic Optimization via Operator Splitting and Homogeneous Self-Dual Embedding},
    journal      = {Journal of Optimization Theory and Applications},
    month        = {June},
    year         = {2016},
    volume       = {169},
    number       = {3},
    pages        = {1042-1068},
    url          = {http://stanford.edu/~boyd/papers/scs.html},
}

@article{huangfuParallelizingDualRevised2018,
  title = {Parallelizing the Dual Revised Simplex Method},
  author = {Huangfu, Q. and Hall, J. A. J.},
  year = {2018},
  month = mar,
  journal = {Mathematical Programming Computation},
  volume = {10},
  number = {1},
  pages = {119--142},
  issn = {1867-2957},
  doi = {10.1007/s12532-017-0130-5},
  urldate = {2025-05-09},
  langid = {english}
}

@misc{dalleDifferentiationInterface2025,
      author={Dalle, Guillaume and Hill, Adrian},
      title={Differentiation{I}nterface.jl},
      year={2024},
      publisher={Zenodo},
      doi={10.5281/zenodo.11092033},
      url={https://doi.org/10.5281/zenodo.11092033},
}

@misc{hill2025sparserbetterfasterstronger,
      title={{Sparser, Better, Faster, Stronger: Efficient Automatic Differentiation for Sparse Jacobians and Hessians}}, 
      author={Adrian Hill and Guillaume Dalle},
      year={2025},
      eprint={2501.17737},
      archivePrefix={arXiv},
      primaryClass={cs.LG},
      url={https://arxiv.org/abs/2501.17737}, 
}

@misc{schäfer2022abstractdifferentiationjlbackendagnosticdifferentiableprogramming,
      title={AbstractDifferentiation.jl: Backend-Agnostic Differentiable Programming in Julia}, 
      author={Frank Schäfer and Mohamed Tarek and Lyndon White and Chris Rackauckas},
      year={2022},
      eprint={2109.12449},
      archivePrefix={arXiv},
      primaryClass={cs.MS},
      url={https://arxiv.org/abs/2109.12449}, 
}

@article{bezansonJuliaFreshApproach2017,
  title = {Julia: {{A Fresh Approach}} to {{Numerical Computing}}},
  shorttitle = {Julia},
  author = {Bezanson, Jeff and Edelman, Alan and Karpinski, Stefan and Shah, Viral B.},
  year = {2017},
  month = jan,
  journal = {SIAM Review},
  volume = {59},
  number = {1},
  pages = {65--98},
  publisher = {{Society for Industrial and Applied Mathematics}},
  issn = {0036-1445},
  doi = {10.1137/141000671},
  urldate = {2025-05-09}
}

@article{Lubin2023,
    author = {Miles Lubin and Oscar Dowson and Joaquim {Dias Garcia} and Joey Huchette and Beno{\^i}t Legat and Juan Pablo Vielma},
    title = {{JuMP} 1.0: {R}ecent improvements to a modeling language for mathematical optimization},
    journal = {Mathematical Programming Computation},
    volume = {15},
    pages = {581–589},
    year = {2023},
    doi = {10.1007/s12532-023-00239-3}
}

@book{boydConvexOptimization2004,
  title = {Convex {{Optimization}}},
  author = {Boyd, Stephen and Vandenberghe, Lieven},
  year = {2004},
  month = mar,
  edition = {1st edition},
  publisher = {Cambridge University Press},
  address = {Cambridge New York Melbourne New Delhi Singapore},
  isbn = {978-0-521-83378-3},
  langid = {english}
}

@article{harrachCalderonProblemFinitely2023,
  title = {The {{Calder{\'o}n Problem}} with {{Finitely Many Unknowns}} Is {{Equivalent}} to {{Convex Semidefinite Optimization}}},
  author = {Harrach, Bastian},
  year = {2023},
  month = oct,
  journal = {SIAM Journal on Mathematical Analysis},
  pages = {5666--5684},
  publisher = {{Society for Industrial and Applied Mathematics}},
  issn = {0036-1410},
  doi = {10.1137/23M1544854},
  urldate = {2023-10-19}
}
